\documentclass[hidelinks,onefignum,onetabnum]{siamart251216}

\usepackage{amsmath, amssymb, bm}
\usepackage{amsfonts}
\usepackage{graphicx}
\usepackage{epstopdf}
\usepackage{booktabs}
\usepackage{float}
\usepackage{tikz}
\usepackage{subcaption}
\usepackage{xcolor}
\usepackage{xurl}
\usepackage{tikz}
\usetikzlibrary{arrows.meta,positioning,calc}
\usepackage{algorithmic}
\ifpdf
  \DeclareGraphicsExtensions{.eps,.pdf,.png,.jpg}
\else
  \DeclareGraphicsExtensions{.eps}
\fi

\graphicspath{{pic/}}

\newsiamremark{remark}{Remark}
\newsiamthm{example}{Example}
\newsiamremark{hypothesis}{Hypothesis}
\crefname{hypothesis}{Hypothesis}{Hypotheses}
\newsiamthm{claim}{Claim}
\newsiamremark{fact}{Fact}
\crefname{fact}{Fact}{Facts}

\headers{Fei Xie, Nan Lu, Yajuan Sun}{Energy stable auxiliary variable}

\title{Energy stable auxiliary variable method for Cahn--Hilliard equations\thanks{Submitted to the editors DATE.
		\funding{The work of the third author is supported by the National Natural Science Foundation of China (Grant No. 12271513).}}}

\author{Fei Xie\thanks{State Key Laboratory of Mathematical Sciences, Academy of Mathematics and Systems Science, Chinese Academy of Sciences, Beijing 100190, China, and University of Chinese Academy of Sciences, Beijing 100049, China (xiefei2021@lsec.cc.ac.cn).}
	\and Nan Lu\thanks{Department of Mathematics, Shanghai Normal University, Shanghai, China 200234 (lunan@shnu.edu.cn).}
	\and Yajuan Sun\thanks{State Key Laboratory of Mathematical Sciences, Academy of Mathematics and Systems Science, Chinese Academy of Sciences, Beijing 100190, China, and University of Chinese Academy of Sciences, Beijing 100049, China (sunyj@lsec.cc.ac.cn).}}

\usepackage{amsopn}

\ifpdf
\hypersetup{
  pdftitle={Energy stable auxiliary variable method},
  pdfauthor={Fei Xie, Nan Lu, Yajuan Sun}
}
\fi

\begin{document}

\maketitle
\begin{abstract}
	In this paper, we propose a quadratic reformulation theory for rational-like functions. Based on this theory, we develop the Quadratic Conservation Elevation (QCE) method, which combines the Scalar Auxiliary Variable (SAV) method with the implicit midpoint rule. We apply this approach to the Cahn–Hilliard (CH) equation with rational-like free-energy terms, obtaining numerical discretizations that preserve the original energy dissipation law. We further derive the discrete dispersion relation and coarsening dynamics, confirming the method's efficiency and consistency with the continuous counterpart. In addition, we use the proposed method to capture missing orientations for  different anisotropic functions.  Numerical  simulations with various   initial conditions illustrate phase separation and anisotropic evolution.	
\end{abstract}



\section{Introduction}
The Cahn--Hilliard (CH) equation is a fundamental phase-field model for phase separation and interface-driven pattern formation in materials science \cite{Boettinger2002,cahn1958}. It can be formulated as an $H^{-1}$ gradient flow associated with a Ginzburg--Landau free energy and exhibits two key structural properties: mass conservation and energy dissipation \cite{Wu2022Review}. To avoid nonphysical energy drift, long-time numerical simulations require schemes that preserve the energy structure inherited from the continuous system. 

A variety of numerical discretizations preserving the energy properties have been developed for CH-type equations, including the discrete variational derivative method \cite{Furihata2010DVD}, the average vector field (AVF) method \cite{Celledoni2012}, and discrete-gradient-based integral-preserving frameworks for PDEs \cite{DahlbyOwren2011}. 
In recent years, auxiliary-variable techniques, including the invariant energy quadratization (IEQ) method \cite{Yang2017IEQ} and the scalar auxiliary variable (SAV) method \cite{Shen2018SAV}, have become widely used for gradient flows. These approaches are highly effective, but the reformulated energy does not always coincide directly with the original one.
More recently, Tapley \cite{Tapley2022} proposed the multiple quadratic auxiliary variable (MQAV) method which can preserve the original first integrals. However, it is only suited to polynomial-type invariants.

In many applications, isotropic interfacial energy is insufficient to reproduce faceted shapes. This motivates anisotropic CH models, in which the interfacial energy depends on the interface orientation represented by the phase-field variable \(u\), typically through its gradient \(\nabla u\) \cite{karma1996,Kobayashi1993}. Such anisotropic surface energies give rise to direction-dependent surface tension and equilibrium shapes characterized by the Wulff construction \cite{CahnHoffman1974,Wulff1901}. Various phase-field formulations and analytical treatments of anisotropic energies, including strongly anisotropic cases with facets and missing orientations, have been proposed in \cite{Eggleston2001,TaylorCahn1998,Torabi2009}.
For anisotropic CH models, the associated energy functional is usually more complicated. Here, we focus on a rational-like energy form, for which the construction of structure-preserving numerical discretizations is particularly challenging.

This paper establishes a quadraticization framework for rational-like energy functionals and provides a theoretical justification for such reformulations. By introducing quadratic auxiliary variables, the original rational-like energy can be rewritten as a quadratic functional in an extended variable space. The extended formulation is consistent with the original one through Casimir-type constraints of the extended system. Based on this reformulation, we apply the implicit midpoint method to the extended system and construct the Quadratic Conservation Elevation (QCE) scheme. We prove that the proposed method satisfies the original discrete energy identity and preserves the corresponding energy-dissipation law. Further numerical analysis verifies the discrete dispersion relation, spinodal decomposition, coarsening behavior, and second-order temporal accuracy for both isotropic and anisotropic CH models.

The outline of this paper is as follows. Section~2 presents the theoretical framework for quadratizing rational-like nonlinear functions and verifies the structure-preserving property of the extended formulation. Section~3 develops fully discrete schemes for isotropic and anisotropic CH equations and proves their discrete energy-dissipation laws. Section~4 examines the discrete dispersion relation and related physical phenomena, including spinodal instability, coarsening behavior, and missing orientations. Section~5 verifies the second-order temporal accuracy and investigates the evolution dynamics from different initial conditions. Finally, Section~6 concludes the paper with a brief summary and several directions for future work.
\section{Quadratization-based numerical methods}
We start this section from introducing the definition of rational-like functions.
\begin{definition}[Rational-like Function]\label{defn:rational}
Consider a  function $H(x)$ defined on $\mathbb{R}^n$,
	it is called the rational-like function if there exists a nonzero polynomial of degree $d$ in $n+1$ variables
	$P(x,y)=\sum_{i=0}^d p_i(x)\,y^i$,
	such that for all $x\in\mathbb{R}^n$,
	\[
	P\bigl(x,\;H(x)\bigr)
	\equiv 0.
	\]
\end{definition}

	From the definition above, it follows that the rational-like function $H(x)$ can be derived from the coordinate functions \(x_1, \dots, x_n\) by a finite sequence of operations, including addition, subtraction, multiplication, and division and the extraction of real \(m\)-th roots. Typical examples include polynomial functions, rational functions, radical functions, etc.
	
\begin{definition}[Dimension-raising Transformation]
	Suppose \(G(x,y):\mathbb{R}^{n+d} \to \mathbb{R}^d\) is a continuously differentiable function, and assume that $\partial_yG(x,y)\in\mathbb{R}^{d\times d}$ is nonsingular. For a given \(H(x):\mathbb{R}^n \to \mathbb{R}\), $G$  is called the dimension-raising function if \(H(x)\) can be embedded into a higher-dimensional function \(\tilde{H}(x,y):\mathbb{R}^{n+d} \to \mathbb{R}\) such that
	\[
	\tilde{H}(x,y) \big|_{G(x,y)=0} = H(x).
	\]
\end{definition}

A function $H:\mathbb{R}^n\to\mathbb{R}$  is   quadraticizable  if for some integer $d\ge 1$, there exists a  map $G:\mathbb{R}^{n+d}\to\mathbb{R}^d$ and a quadratic function $\widetilde H:\mathbb{R}^{n+d}\to\mathbb{R}$ satisfying
$\widetilde H(x,y)=H(x)$ for $(x,y)\in \mathcal{M}$. Here,   $\mathcal{M}:=\{(x,y)\in\mathbb{R}^{n+d}\mid G(x,y)=0\}$.
The process is called quadratic dimensional-raising.
In the following, we show that any rational-like function can be mapped to a quadratic function on  higher-dimensional space via a quadratic dimension-raising transformation.
\begin{theorem}\label{thm:quar}
	Assume that  $H:\mathbb{R}^n\to\mathbb{R}$ be a rational-like function, then
	it can be quadraticized.
\end{theorem}
\begin{proof}
	Due to the definition of the rational-like function, we consider the following four cases. 
	
	Firstly, assume that $H(x)$ is a monomial of $n$-variable with even degree $d=2k$. By relabeling each occurrence of a variable as a distinct symbol, we write every repeated factor $x_i$ as $x_{\alpha_j}$ with a new index $\alpha_j$.
	Then $H(x)$ can be expressed by
	\[
	H(x)=x_{\alpha_1}x_{\alpha_2}\cdots x_{\alpha_{2k-1}}x_{\alpha_{2k}},
	\]
	where $\alpha_{i}, i=1,\cdots, 2k$ takes value from 1 to $2k$.
	Let	$y_{i}=x_{\alpha_{2i-1}}\,x_{\alpha_{2i}}$, 
	we can define a dimension-raising function with $i$-th element $G_i(x, y)=y_{i}-x_{\alpha_{2i-1}}\,x_{\alpha_{2i}}$, which is quadratic. Then $\tilde{H}(x,y)=H(x)= y_1 \cdots y_k$ is reduced to monomial of $n$-variable with degree $k$. If $k>1$, we apply the same procedure to the product $\prod_{i=1}^k y_i$.
		
	The second case considers a monomial of odd degree $d=2k+1$, namely,
	\[
	H(x)=x_{\alpha_1}\cdots x_{\alpha_{2k}}\,x_{\alpha_{2k+1}}
	=\Bigl(\prod_{j=1}^{2k}x_{\alpha_j}\Bigr)\,x_{\alpha_{2k+1}}.
	\]
	The first term is a monomial of  degree $2k$. Thus, using the same auxiliary variable  as  the first case,   we can define  $\tilde{H}(x,y)=
	H(x)=\Bigl(\prod_{i=1}^k y_i\Bigr)\,x_{\alpha_{2k+1}}$.
	After finitely many steps, the original monomial is represented by a quadratic extended function subject to quadratic constraints, and hence $H$ can be quadraticized.
	
	In the third case, suppose \(H(x)\) is a function with a rational power. If \(H(x)=1/x\), introducing auxiliary variable \(y=1/x\) gives \(\tilde{H}(x, y)=y\). The corresponding dimension-raising quadratic function is \(G(x,y) = xy - 1\). 

%
In the last case, consider $H(x)$ has the rational power $H(x)=x^{1/q}$ with $q\in\mathbb{N}$, $q\ge2$. From the binary expansion we have $q=\sum_{i=0}^{s} b_i 2^i$. We introduce auxiliary variable $u_j$, $v_j$, $i=0,1,\cdots s$  intend to satisfy $u_0^q=x$. The procedure of constructing the variables are as follows:

%
%
%
%
%
%
%
 
 \begin{algorithm}[htbp]
 	\caption{Construction of auxiliary variables for $x^{1/q}$}
 	\label{alg:recursive-auxiliary}
 	\begin{algorithmic}[1]
 		\STATE {\bf Input:} Binary coefficients $b_0,\dots,b_s$ such that
 		$q=\sum_{j=0}^s$
 		\STATE {\bf Output:} Auxiliary variables $u_0,\dots,u_s$ and $v_0,\dots,v_s$
 		\STATE Set $(u_0=x^{1/q})$, $v_0=1$
 		\FOR{$i=0,1,\dots,s-1$}
 		\STATE $u_{i+1}= u_i^2$
 		\STATE $v_{i+1}= v_i\,u_i^{\,b_i}$
 		\ENDFOR
 		\STATE  $x=v_su_s^{\,b_s}$
 	\end{algorithmic}
 	\end{algorithm}
Once $u_0=x^{1/q}$ is enforced by the quadratic constraints, the general case $x^{p/q}=(u_0)^p$ follows, and the monomial construction applies to $(u_0)^p$. This concludes the proof.
\end{proof}

Consider the ODE system in the form of
\begin{align}\label{eq:ODE}
	\dot x= f(x).
\end{align}
If the system has the conservative or dissipative properties, then it can be rewritten as
\begin{align}\label{eq:gradientODE}
	\dot x=\mathbb{M}(x)\nabla H(x)
\end{align}
with $\mathbb{M}$ is skew-symmetric or negative semi-definite\cite{McLachlanQuispelRobidoux1999}.


For system~\eqref{eq:ODE}, Theorem~\ref{thm:quar} shows that if the quantity $H(x)$ is rational-like, it can be quadraticized and embedded into an equivalent system defined on  extended space \( \tilde{H} \).
However, the dimension-raising transformation \( G(x,y) \) is not unique. Under the condition of the implicit function theorem, it is known that there exists a smooth function \(\varphi:\mathbb{R}^n\to\mathbb{R}^d\),
such that $y = \varphi(x)$  is implicitly determined  by  $G\bigl(x,\varphi(x)\bigr)\equiv 0$.

Introduce the homeomorphic  map $j:\mathbb{R}^n\to M, j(x)=(x,\varphi(x))$.
It is clear that
\[
\tilde H(x,y)\big|_M
= \tilde H\bigl(j(x)\bigr)
= H(x).
\]
In what follows, we demonstrate that this extended system not only exactly preserves the original invariants but also has a variety of geometric and algebraic properties.

Suppose the quadratic dimension-raising transformation $G(x,y)$ is $d$-dimensional, then without loss of generality it can be expressed as

Suppose each dimension-raising transformation $G_i(x,y)$ is a homogeneous quadratic polynomial of $d$-dimensional. Let  $z=(x^\top,y^\top)^\top\in\mathbb{R}^{n+d}$. Then $G_i(x,y)=\frac12\,z^\top Q^i z,\; i=1,\ldots,d,$ where $Q^i\in\mathbb{R}^{(n+d)\times(n+d)}$ is symmetric and can be partitioned as
\[
Q^i=\begin{pmatrix}
	Q_{11}^i & Q_{12}^i\\
	(Q_{12}^i)^\top & Q_{22}^i
\end{pmatrix}.
\]
The Jacobian $DG(x,y)\in\mathbb{R}^{d\times(n+d)}$ is given by
\[
\begin{aligned}
	D G&=
	\begin{pmatrix}
		\nabla G_1^\top \\
		\vdots  \\[0.75ex]
		\nabla G_d^\top\\
	\end{pmatrix}
	&=\left(\underbrace{
		\begin{array}{c}
			x^\top Q_{11}^1 + y^\top Q_{12}^{1\,\top} \\[0.75ex]
			\vdots \\[0.75ex]
			x^\top Q_{11}^d + y^\top Q_{12}^{d\,\top}
		\end{array}
	}_{N(x,y)}
	\;\Bigg|\;
	\underbrace{
		\begin{array}{c}
			x^\top Q_{12}^1 + y^\top Q_{22}^1 \\[0.75ex]
			\vdots \\[0.75ex]
			x^\top Q_{12}^d + y^\top Q_{22}^d
		\end{array}
	}_{M(x,y)}\!\right).
\end{aligned}
\]

For a given system~\eqref{eq:ODE}, the corresponding extended system is
\begin{equation}\label{eq:extendODE}
	\begin{pmatrix}
		\dot{x}\\
		\dot{y}
	\end{pmatrix}
	=
	F(x,y)
	:=
	\begin{pmatrix}
		\tilde f(x,y)\\[0.2em]
		-\,M^\dagger(x,y)\,N(x,y)\,\tilde f(x,y)
	\end{pmatrix},
\end{equation}
where $M^\dagger(x,y)$ denotes the Moore--Penrose pseudoinverse of $M(x,y)$.
If $M(x,y)$ has full column rank, then
\[
M^\dagger(x,y)=\bigl(M(x,y)^\top M(x,y)\bigr)^{-1}M(x,y)^\top.
\]

Assume that $\mathcal{M}$ can be represented as a graph $y=\varphi(x)$. Then the restriction of $F$ to $\mathcal{M}$ is
$F(x,y)\big|_M=F(x,\varphi(x))=F(j(x))$. Since the $x$-component of the extended system~\eqref{eq:extendODE} is $\dot x=\tilde f(x,y)$, consistency with the original system
$\dot x=f(x)$ on $\mathcal{M}$ implies
\begin{equation}\label{eq:dotx}
	\tilde f\bigl(j(x)\bigr)=f(x).
\end{equation}
Moreover, along trajectories on $M$ we have $y(t)=\varphi(x(t))$, and the chain rule gives $\dot y(t)=\varphi'\bigl(x(t)\bigr)\,\dot x(t)=\varphi'\bigl(x(t)\bigr)\,f\bigl(x(t)\bigr)$. On the other hand, the $y$-component of system~\eqref{eq:extendODE} is
$\dot y=-M^\dagger(x,y)\,N(x,y)\,\tilde f(x,y)$, and therefore, on $\mathcal{M}$,
\begin{equation}\label{eq:doty}
	-\,M^\dagger\bigl(j(x)\bigr)\,N\bigl(j(x)\bigr)\,f(x)
	=\varphi'(x)\,f(x).
\end{equation}

Combining Eq.~\eqref{eq:dotx} and Eq.~\eqref{eq:doty} yields
\begin{equation}\label{eq:F_restriction}
	F(j(x))=D(j(x))\,f(x)
	=
	\begin{pmatrix}
		I\\
		\varphi'(x)
	\end{pmatrix}f(x).
\end{equation}
Here $D(j(x))$ denotes the Jacobian matrix of $j(x)$.

For a dynamical system~\eqref{eq:ODE}, a scalar function \(C(z)\) is called a Casimir invariant if it remains constant along every trajectory of the system, which is,
\[
\frac{d}{dt}C(x(t))=\nabla C(x)\cdot f(x)=0.
\]

In our setting, the constraints \(G(x,y)=0\) may be viewed as Casimir functions of the extended system~\cite{Lu2025}. It is not sufficient to preserve only the modified energy \(\widetilde H(x,y)\) in the extended space. To recover the original energy law, one must also preserve the Casimir constraints \(G(x,y)=0\), so that the numerical solution remains on the constraint manifold.

\begin{proposition}
The extended system~\eqref{eq:extendODE} is equivalent to the original system. Moreover, it inherits the same conservative (or dissipative) property as the original system.
In addition, the constraints $G(x,y)$ are Casimir invariants of the extended dynamics.
\end{proposition}

\begin{proof}
	Let \(F(x,y)\) denote the vector field of the extended system. On the constraint manifold \(\mathcal{M}\), \(\tilde H(x,y)\) coincides with the original energy \(H(x)\), and the extended dynamics reduces to the original one in the \(x\)-component. Hence, by construction,
	\begin{equation}\label{eq:verify_pre}
		\nabla \tilde H(x,y)^\top F(x,y)\big|_{\mathcal{M}}
		=
		\nabla H(x)^\top f(x).
	\end{equation}
	Therefore,
	\[
	\nabla \tilde H(x,y)^\top F(x,y)\big|_{\mathcal{M}}
	=
	\begin{cases}
		0, & \text{if the original system is conservative},\\[0.4em]
		\le 0, & \text{if the original system is dissipative}.
	\end{cases}
	\]
	Hence, the extended system inherits the same conservation or dissipation property on \(\mathcal{M}\).
	
	The equivalence between the extended system and the original system on the constraint manifold follows from \cite{Lu2025}.
\end{proof}
In fact, by introducing the Casimir function, the energy function of the dimension-raised system becomes quadratic with respect to the extended variables, while the equivalence with the original system holds pointwise under the Casimir constraint ~\cite{Lu2025}. Applying the implicit midpoint rule in time then yields a scheme that preserves the energy-preserving or energy-dissipasive law. Moreover, the temporal discretization is not restricted to the midpoint rule, but can be extended to more general symplectic Runge--Kutta methods.

Here, we discretize the extended system~\eqref{eq:extendODE} by the implicit midpoint scheme, in the following theorem we prove that the resulting numerical discretization  can preserve the original conservative properties.

\begin{proposition}
	Applying the implicit midpoint method to the extended system~\eqref{eq:extendODE} yields a second-order discretization that exactly preserves the quadratic invariants of the extended system. In particular, if the initial data satisfy the constraint \(G(x_0,y_0)=0\), then the numerical solution remains on the constraint manifold for all time steps. After eliminating the auxiliary variables through the constraint relation, the resulting reduced method is a second-order integrator that exactly preserves the original invariant inherited from the continuous system.
\end{proposition}
\begin{proof}
It is known that the midpoint method preserves the quadratic invariants of system, thus  $G(x, y)$  can preserved, i.e., $G(x_n, y_n)=G(x_{n+1}, y_{n+1}) = 0$. Furthermore, the original energy of system can be preserved exactly.

Due to the implicit function theorem, it is easy to get \(y=\varphi(x)\) at the grid point from \(G(x,y)=0\).
Substituting \(y\) at grid point into midpoint discretization of system~\eqref{eq:extendODE}, leads to an invariant preserving numerical algorithm:
\[
\frac{x_{n+1} - x_n}{h} = \widetilde{f}\left( \frac{x_{n} + x_{n+1}}{2}, \frac{\varphi(x_n) + \varphi(x_{n+1})}{2} \right).
\]
Obviously, it is of order 2 by Taylor expansion in \(h\).
\end{proof}

The above analysis can be extended to the PDE setting. Consider an evolutionary system of the form
\begin{equation}\label{eq:ori_pde}
	u_t \;=\; \mathcal{D}\,\frac{\delta \mathcal{H}}{\delta u},
\end{equation}
where \(\mathcal{D}\) is a differential operator. Let $\operatorname{Pr}^{(n)}u=\bigl(u,\partial_x u,\partial_x^2u,\dots,\partial_x^n u\bigr)$ denote the \(n\)-th prolongation of \(u\). The Hamiltonian functional is defined by
\begin{equation}\label{eq:H_pde}
	\mathcal{H}[u]
	=
	\int_{\Omega} H\bigl(\operatorname{Pr}^{(n)}u\bigr)\,\mathrm{d}x,
\end{equation}
where \(\Omega\subset\mathbb{R}^d\) and \(\mathrm{d}x=\mathrm{d}x_1\cdots \mathrm{d}x_d\).

It follows that
\[
\frac{d}{dt}\mathcal{H}
=
\left( u_t,\, \frac{\delta \mathcal{H}}{\delta u} \right)
=
\left( \mathcal{D}\frac{\delta \mathcal{H}}{\delta u},\, \frac{\delta \mathcal{H}}{\delta u} \right)
\begin{cases}
	=0, & \text{if } \mathcal{D} \text{ is skew-adjoint},\\[0.4em]
	\le 0, & \text{if } \mathcal{D} \text{ is negative semidefinite}.
\end{cases}
\]
Here, \(\delta \mathcal{H}/\delta u\) denotes the variational derivative of \(\mathcal{H}\), and \((\cdot,\cdot)\) denotes the inner product.

\begin{definition}\label{defi:1}
	For the system~\eqref{eq:ori_pde}, the energy functional \(\mathcal{H}[u]\) is called quadratic if its density \(H(\operatorname{Pr}^{(n)}u)\) is a quadratic form in \(\operatorname{Pr}^{(n)}u\), namely,
	\[
	H(\operatorname{Pr}^{(n)}u)
	=
	\frac12\bigl(\operatorname{Pr}^{(n)}u\bigr)^\top
	A\,
	\operatorname{Pr}^{(n)}u,
	\]
	where \(A\) is a symmetric constant matrix.
\end{definition}
For such functionals the implicit midpoint method inherits a precise structure preserving property.
\begin{theorem}
	If \(\mathcal{H}[u]\) is quadratic, then the implicit midpoint discretization of system~\eqref{eq:ori_pde} satisfies
	\[
	\mathcal{H}[u_{n+1}] - \mathcal{H}[u_n]
	=
	\Delta t\,\bigl(\mathcal{D}Su^{n+\frac12},\,Su^{n+\frac12}\bigr).
	\]
	In particular,
	\[ 
	\bigl(DSu^{n+\frac12},\,Su^{n+\frac12}\bigr) \begin{cases} =0, & \text{if } D \text{ is skew-adjoint},\\[0.4em] \le 0, & \text{if } D \text{ is negative semidefinite}. \end{cases} 
	\]
\end{theorem}
\begin{proof}
	Applying the implicit midpoint rule to system~\eqref{eq:ori_pde} gives
	\[
	\frac{u_{n+1}-u_n}{\Delta t}
	=
	\mathcal{D}\left.\frac{\delta \mathcal H}{\delta u}\right|_{u^{n+\frac12}},
	\qquad
	u^{n+\frac12}:=\frac{u_n+u_{n+1}}{2}.
	\]
	Since \(\mathcal H[u]\) is quadratic, its variational derivative depends linearly on \(u\). Hence there exists a symmetric linear operator \(S\) such that
	\[
	\frac{\delta \mathcal H}{\delta u}=Su.
	\]
	Therefore, ${(u_{n+1}-u_n)}/{\Delta t}=\mathcal{D}Su^{n+\frac12}.$
	Multiplying this with $\mathcal{S} u$ leads to
	\[
	\mathcal H[u_{n+1}]-\mathcal H[u_n]
	=\bigl(u_{n+1}-u_n,\; Su^{n+\frac12}\bigr)
	=\Delta t\,\bigl(\mathcal{D}Su^{n+\frac12},\,Su^{n+\frac12}\bigr).
	\]
	The conclusion follows immediately.
\end{proof}
\begin{theorem}
Assume that the following PDE
\begin{equation}\label{eq:pde}
	u_t={\mathcal P}(\operatorname{Pr}^{(n)}u)
\end{equation}
has a rational-like density function. Then, there exist auxiliary variables
$\mathbf y(x)\in \mathbb{R}^m$ determined by the dimension-raising functions $G_i\left(\operatorname{Pr}^{(n)}u,{\mathbf y}\right)$
such that on extended space $(\operatorname{Pr}^{(n)}u,{\mathbf y})$
we can get  a new system with quadratic functional
\[
\widetilde{\mathcal H}[u,\mathbf{y}]
=\int_{\Omega}
\widetilde H\left(\bigl(\operatorname{Pr}^{(n)}u, \mathbf{y}\right)\,\mathrm{d}x,
\]
and it satisfies $\widetilde{\mathcal H}[u,\mathbf{y}]=\;\mathcal H[u].$
\end{theorem}

\begin{proof}
Denote $u_i$ be the \(i\)-th component of \(\operatorname{Pr}^{(n)}u\), then $\operatorname{Pr}^{(n)}u=(u_0,u_1,\cdots, u_n)$ with $u_0=u$. The rational-like density $H(\operatorname{Pr}^{(n)}u)$ has the same structure as in Theorem~\ref{thm:quar}. Therefore, we apply the same quadratization procedure to obtain a quadratic extended representation.
\end{proof}

On the extended space, the system yields
\[
\left\{
\begin{aligned}
		&u_t = D\left(\frac{\partial \widetilde{H}}{\partial u}
		+\sum_{i=1}^n (-1)^i D_x^{(i)}\!\left( \frac{\partial \widetilde{H}}{\partial u^{(i)}} \right)
		+\frac{\partial \widetilde{H}}{\partial y}\frac{\partial y}{\partial u}
		+\sum_{i=1}^n (-1)^i D_x^{(i)}\!\left( \frac{\partial \widetilde{H}}{\partial y}\frac{\partial y}{\partial u^{(i)}} \right) \right)\\
		&\frac{\partial G}{\partial y} y_t=-\left(\frac{\partial G}{\partial u}u_t
		+ \sum_{i=1}^n(-1)^i D_x^{(i)}\!\left( \frac{\partial G}{\partial u_x^{(i)}} u_t \right)\right).
\end{aligned}\right.\]
Then when the system is discretized in time using the midpoint rule, both the quadratic invariant and the corresponding weak invariants are exactly preserved.

This section has developed Quadratic Conserving Elevation (QCE) method. Its main idea is to quadraticize the nonlinear energy through suitable auxiliary variables, leading to an equivalent extended system with quadratic Casimir quantities.
The midpoint discretization of the extended system, after eliminating the auxiliary variables, preserves the original conservative or dissipative structure.
\section{Numerical discretization for CH system}
The Cahn-Hilliard (CH) equation \cite{cahn1958} describes the dynamics of phase separation, driven by a diffuse-interface free energy.
The isotropic CH equation reads
\begin{equation}\label{eq:CH-iso}
	\left\{
	\begin{aligned}
		&u_t=\nabla\!\cdot\!\big(M\nabla\mu\big),
		\qquad (x,t)\in \Omega \times (0,T],\\
		&\mu=-\Delta u+\varepsilon^{-2}F'(u)+\beta\,\Delta^2 u,
	\end{aligned}
	\right.
\end{equation}
where \(0<\varepsilon\ll1\) is the interfacial thickness, \(\beta\ge0\) is the regularization strength, $F(u)$ is the bulk potential.
Its free energy functional is given by
\begin{equation}\label{eq:CH-iso-e}
	E_{\mathrm{iso}}(u)
	=\frac12(\nabla u,\nabla u)+\varepsilon^{-2}(F(u),1)+\frac{\beta}{2}(\Delta u,\Delta u),
\end{equation}
where \((\cdot,\cdot)\) denotes the \(L^2(\Omega)\) inner product. namely, \((\phi,\psi)=\int_\Omega \phi(x)\psi(x)\,\mathrm{d}x\). It is known that
\begin{align}\label{eq:EnDis}
	\frac{\mathrm d}{\mathrm d t}E_{\mathrm{iso}}(u)
	=\int_\Omega \mu\,u_t\,\mathrm d x
	=\int_\Omega \mu\,\nabla\!\cdot\!\bigl(M\nabla\mu\bigr)\,\mathrm d x
	=-M\int_\Omega |\nabla\mu|^2\,\mathrm d x\le 0
\end{align}
which indicates the system is dissipative. For Eq.\eqref{eq:CH-iso}, we apply the QCE method in time and the Fourier pseudo-spectral discretization in space. In this section, we show that the resulting numerical discretizations can preserve the dissipative property of the CH equation.

Let $\Omega=[a_x,b_x]\times[a_y,b_y]$, $h_x$ and $h_y$ be the grid sizes in the $x$ and $y$ directions, resp.
Set $\Omega_h
=
\{(x_j,y_k)| x_j=a_x+jh_x,\;
y_k=a_y+kh_y\}$ and define $V_h=\{U|U=\{U_{jk}\},(x_j,y_k)\in\Omega_h \}$
as the space of grid functions on $\Omega_h$. Denote \(D_x\) and \(D_y\) as the first-order Fourier pseudo-spectral differentiation matrices \cite{Lu2024}. 
Take \(M=1\) and $F(u)=\frac14(u^2-1)^2$, we have \(F'(u)=u^3-u\). Employing the Fourier pseudo-spectral
method for system~\eqref{eq:CH-iso} in space gives
\begin{equation}\label{eq:CH-iso-semi-dis}
	U_t
	-\Delta_h\Bigl[
	-\Delta_hU+\varepsilon^{-2}(U^{\odot 3}-U)+\beta\Delta_h^2U
	\Bigr]=0,
\end{equation} 
for $U\in V_h$. Here,   $\Delta_h$  denotes the discrete Laplacian which is  $\Delta_h U = D_x^2U + U(D_y^\top)^2$,  and \(\odot\) denotes the Hadamard product.\footnote{For any $U, V\in V_h$, the Hadamard product is defined by $(U\odot V)_{ij}=U_{ij}V_{ij}$, $U^{\odot 2}=U\odot U$.}

By introducing the auxiliary variable $q=\frac12(u^2-1)$,
we apply the QCE method to system~\eqref{eq:CH-iso-semi-dis}  to obtain the following fully-discrete formulation:
\begin{equation}\label{fu-dis-scheme:CH-iso}
	\left\{
	\begin{aligned}
		\delta_t U^{n+1} &= \Delta_h \mu^{n+\frac12},\\
		\mu^{n+\frac12}
		&=
		-\Delta_h U^{n+\frac12}
		+2\varepsilon^{-2}U^{n+\frac12}\odot Q^{n+\frac12}
		+\beta \Delta_h^2 U^{n+\frac12},\\
		\delta_t Q^{n+1} &= U^{n+\frac12}\odot \delta_t U^{n+1},
	\end{aligned}
	\right.
\end{equation}
where $\delta_t Z^n=\frac{Z^{n+1}-Z^n}{\Delta t},Z^{n+\frac12}=\frac{Z^{n+1}+Z^n}{2}$ and $Q=\frac12\bigl(U^{\odot2}-1\bigr)$. Define the discrete energy corresponding to \eqref{eq:CH-iso-e} as
\begin{equation}\label{eq:E-iso-h}
	E_{\mathrm{iso},h}(U)
	=
	-\frac12(\Delta_hU,U)_h
	+\varepsilon^{-2}(F(U),\mathbf1)_h
	+\frac{\beta}{2}\|\Delta_hU\|_h^2,
\end{equation}
where $(U,V)_h=h_xh_y\sum U_{jk}V_{jk}$ is the discrete inner product on $V_h$,  $\|U\|_h^2=(U,U)_h$, and \(F(U)\) is defined pointwise, i.e., \((F(U))_{jk}=F(U_{jk})\). Here, \(\mathbf1\) denotes the grid function with all entries equal to \(1\). Then, we have the following theorem.

\begin{theorem}\label{thm:CH-iso-e-dis}
	For the isotropic CH equation \eqref{eq:CH-iso}, the numerical discretization ~\eqref{fu-dis-scheme:CH-iso}  can preserve the dissipative  property of  system, which is
	\[ E_{\mathrm{iso},h}(U^{n+1}) \le E_{\mathrm{iso},h}(U^n) \]
	with  $E_{\mathrm{iso},h}$ defined in Eq.~\eqref{eq:E-iso-h}.
\end{theorem}

\begin{proof}
	Note that $Q=\frac12\bigl(U^{\odot2}-1\bigr)$, the discrete energy can be rewritten as
\begin{equation}\label{eq:e-iso-uq}
	\begin{aligned}
		E_{\mathrm{iso},h}(U)
		&=
		-\frac12(\Delta_hU,U)_h+\varepsilon^{-2}\|Q\|_h^2+\frac{\beta}{2}\|\Delta_hU\|_h^2.
	\end{aligned}
\end{equation}
Noting $\delta_t Q^{n+1}=U^{n+\frac12}\odot \delta_t U^{n+1}$, it follows from~\eqref{eq:e-iso-uq} that
\begin{equation}\label{eq:iso-energy-diff-detail}
	\begin{aligned}
		&\frac{E_{\mathrm{iso},h}(U^{n+1})-E_{\mathrm{iso},h}(U^n)}{\Delta t}\\
		={}&
		\bigl(-\Delta_hU^{n+\frac12},\delta_t U^{n+1}\bigr)_h
		+2\varepsilon^{-2}\bigl(U^{n+\frac12}\odot Q^{n+\frac12},\delta_t U^{n+1}\bigr)_h
		+\beta\bigl(\Delta_h^2U^{n+\frac12},\delta_t U^{n+1}\bigr)_h\\
		={}&
		\bigl(\mu^{n+\frac12}, \delta_t U^{n+1}\bigr)_h.
	\end{aligned}
\end{equation}
	By the first and second identities in Eq.~\eqref{fu-dis-scheme:CH-iso} and discrete integration by parts, we obtain
\[\frac{E_{\mathrm{iso},h}(U^{n+1})-E_{\mathrm{iso},h}(U^{n})}{\Delta t}=
-\|D_x\mu^{n+\frac12}\|_h^2-\|\mu^{n+\frac12}D_y^\top\|_h^2 \\
\le{} 0.\]
This completes the proof.
\end{proof}

The anisotropic Cahn-Hilliard equation extends the classical Cahn-Hilliard model by incorporating directional dependence in the interfacial energy \cite{ma2006implementation}. It can be described as
\begin{equation}\label{eq:CH-ani}
	\left\{
	\begin{aligned}
		&u_t = \nabla\!\cdot\!\big(M\nabla\mu\big),
		\qquad (x,t)\in \Omega \times (0,T],\\
		&\mu=
		-\,\nabla\!\cdot\big(\Gamma(\theta)\,\nabla u\big)
		+\frac{\Gamma(\theta)}{\varepsilon^2}F'(u)
		+\beta\,\Gamma(\theta)\,\Delta^2 u,
	\end{aligned}
	\right.
\end{equation}
where $\Gamma(\theta)$ is the anisotropy factor with $\theta=\operatorname{atan2}(u_y,u_x)$.
The corresponding energy can be written as
\begin{equation}\label{eq:E-ani}
	E_{\mathrm{ani}}(u)
	=
	\frac12 (\nabla u,\nabla u)_\Gamma
	+\varepsilon^{-2}(F(u),1)_\Gamma
	+\frac{\beta}{2}(\Delta u,\Delta u)_\Gamma,
\end{equation}
where $(\phi,\psi)_\Gamma$ denotes the $\Gamma$-weighted $L^2(\Omega)$ inner product defined by $(\phi,\psi)_\Gamma =\int_\Omega \Gamma(\Theta)\,\phi\,\psi\,\mathrm{d}x$.
Clearly, the energy is dissipative:
\[
\frac{\rm d}{\rm dt}E_{\mathrm{ani}}(u)=-M\int_\Omega |\nabla\mu|^2\,\rm d x\le 0.
\]

For $U\in V_h$, denote  $\Gamma(\Theta)=\Gamma(\Theta(U_{jk}))$  with $\Theta(U_{jk})=\operatorname{atan2}((U_{jk})_y,(U_{jk})_x)$ the discrete orientation.
Applying the Fourier pseudo-spectral method in space to~\eqref{eq:CH-ani}, yields the semi-discrete numerical discretization
\begin{equation}\label{eq:semi-discrete-ani}
	U_t-\Delta_h\Psi(U)=0,
\end{equation}
where $\Psi(U)
=-D_x\,\bigl( \Gamma(\Theta)\odot D_xU\bigr)
-\bigl( \Gamma(\Theta)\odot (UD_y^\top)\bigr)D_y^\top+\varepsilon^{-2} \Gamma(\Theta)\odot f(U)
+\beta\, \Gamma(\Theta)\odot \Delta_h^2U.$
Let the discrete weighted norm be $|U|_{1,h,\Gamma}^2=
\bigl( \Gamma(\Theta)\odot U_x,\,U_x\bigr)_h
+
\bigl( \Gamma(\Theta)\odot U_y,\,U_y\bigr)_h,$
and $\|\Delta_h U\|_{h,\Gamma}^2=\bigl( \Gamma(\Theta)\odot \Delta_h U,\,\Delta_h U\bigr)_h$, then
the discrete anisotropic energy is defined by
\begin{equation}\label{eq:E-ani-h}
	E_{\mathrm{ani},h}(U)
	=
	\frac12 |U|_{1,h,\Gamma}^2
	+\varepsilon^{-2}\bigl( \Gamma(\Theta)\odot F(U),\,\mathbf 1\bigr)_h
	+\frac{\beta}{2}\|\Delta_h U\|_{h,\Gamma}^2.
\end{equation}
For the anisotropic case, using the QCE method to \eqref{eq:semi-discrete-ani} to gain the following fully discrete scheme
\begin{equation}\label{eq:CH-ani-aux}
	\left\{
	\begin{aligned}
		\delta_t U^{n+1} &= \Delta_h \mu^{n+\frac12},\\
		\delta_t Y_1^{n+1}
		&=2\,U_x^{n+\frac12}\odot \delta_t U_x^{n+1},\\
		\delta_t Y_2^{n+1}
		&=2\,U_y^{n+\frac12}\odot \delta_t U_y^{n+1},\\
		\delta_t Y_3^{n+1}
		&=2\bigl(Y_1^{n+\frac12}+Y_2^{n+\frac12}\bigr)
		\odot \bigl(\delta_t Y_1^{n+1}+\delta_t Y_2^{n+1}\bigr),\\
		\delta_t Y_4^{n+1}
		&=-\,Y_4^{n+1}\odot Y_4^n\odot \delta_t Y_3^{n+1},\\
		\delta_t Y_5^{n+1}
		&=2\,Y_1^{n+\frac12}\odot \delta_t Y_1^{n+1}
		+2\,Y_2^{n+\frac12}\odot \delta_t Y_2^{n+1},\\
		\delta_t Y_6^{n+1}
		&=Y_4^{n+\frac12}\odot \delta_t Y_5^{n+1}
		+Y_5^{n+\frac12}\odot \delta_t Y_4^{n+1},\\
		\delta_t Z_1^{n+1}
		&=U^{n+\frac12}\odot \delta_t U^{n+1},\\
		\delta_t Z_2^{n+1}
		&=2\,Z_1^{n+\frac12}\odot \delta_t Z_1^{n+1},\\
		\delta_t \Phi_1^{n+1}
		&=2\,\Delta_hU^{n+\frac12}\odot \Delta_h(\delta_t U^{n+1}),
	\end{aligned}
	\right.
\end{equation}
The more detail can be seen in Appendix~\ref{pf:ch-iso-aux}. The following theorem shows that the proposed scheme preserves the energy dissipation.
\begin{theorem}\label{thm:ani-diss-law}
	For the  anisotropic CH equation \eqref{eq:CH-ani}, the numerical discretization
	\eqref{eq:CH-ani-aux} can preserve the dissipative property of system
	\[ E_{\mathrm{ani},h}(U^{n+1}) \le E_{\mathrm{ani},h}(U^n), \]
	where $E_{\mathrm{ani},h}$ is defined in Eq.~\eqref{eq:E-ani-h}.
\end{theorem}

A proof is provided in Appendix~\ref{pf:ani-diss-law}. 

Using the above numerical scheme,  we employ Mann iteration \cite{Mann1953}:
\begin{equation}\label{eq:relaxed_fp} X^{(m+1)}=(1-\omega)X^{(m)}+\omega T\!\bigl(X^{(m)}\bigr), \qquad m=0,1,2,\dots, \end{equation}
where \(\omega \in (0,1]\). 
We next show that the iteration
~\eqref{eq:relaxed_fp} 
converges.
Let $X=(u,y(u))$ collect the original the auxiliary variables. Consider the differential system of the form
\begin{equation}\label{eq:lin-non-sys}
	X_t = \mathcal{L} X + \mathcal{N}(X).
\end{equation}
We use Frouier pseudo-spectral method in spatial.
Let \(X^n\) be the approximation at time level \(t_n\). The Fourier transform of a function \(f\) is denoted by \(\widehat{f}\), \(\mathcal F^{-1}\) denotes the inverse Fourier transform, and define the map \(T\)
\begin{equation}\label{eq:Tdef_once}
	T(X)
	=
	\mathcal{F}^{-1}\!\left[
	\bigl(I-\tfrac12\Delta t\,\widehat{\mathcal L}\bigr)^{-1}
	\left(
	\bigl(I+\tfrac12\Delta t\,\widehat{\mathcal L}\bigr)\widehat{X^n}
	+\Delta t\,\widehat{\mathcal N\!\left(\frac{X+X^n}{2}\right)}
	\right)
	\right].
\end{equation}

The following result gives a sufficient condition for the convergence of the iteration.

\begin{theorem}\label{thm:mann_midpoint}
	Assume that all eigenvalues of \(\mathcal L\) are non-positive, and \(\mathcal N\) is Lipschitz continuous on a nonempty closed convex set \(\mathcal C\) with Lipschitz constant \(L_{\mathcal N}\). If \(\Delta t\,L_{\mathcal N}<2\). Then the map \(T\) has a unique fixed point \(X^\ast\in\mathcal C\).
\end{theorem}

The proof is provided in Appendix~\ref{pf:mann_midpoint}.

For the anisotropic model with fourfold anisotropy, the coefficient
\(\Gamma(\theta(u))\) depends on the auxiliary variable \(y_6 \ge 0\) through $\Gamma(\theta(u))=1-3\alpha+4\alpha\,y_6(u)$. Thus, $\Gamma(\theta(u))\ge 1-3\alpha>0$
for $0<\alpha<\frac13$. Then \(\mathcal L\) is negative semidefinite under periodic boundary conditions after froze linear operator. Also, \(\mathcal N\) is locally Lipschitz on 
on bounded closed convex set \(\mathcal C\) since the auxiliary-variable coefficients
are smooth functions of \(u\). Therefore, the assumptions of Theorem~\ref{thm:mann_midpoint} hold in the anisotropic case.

\section{Numerical analysis for CH system}
We now apply the above framework to the CH system and examine its analytical properties. For the CH equation, the bulk potential \(F(u)\) usually takes a double-well form, e.g.,
$$
F(u)=\tfrac14(u^2-1)^2.
$$
Other common choices include the Flory Huggins form \cite{flory1953principles, huggins1941solutions}
$$
F(u)=(1+u)\ln(1+u)+(1-u)\ln(1-u)+\theta u^2
$$
and the form of higher-order polynomial \cite{wise2009energy}
$$
F(u)=\tfrac{a}{2}u^2+\tfrac{b}{4}u^4+\tfrac{c}{6}u^6 .
$$
The three commonly used potential functions are shown in Fig.~\ref{fig:potential}, the red dots mark the inflection point.
\begin{figure}[htbp]
	\centering
	\begin{subfigure}[t]{0.33\textwidth}
		\centering
		\includegraphics[width=\linewidth]{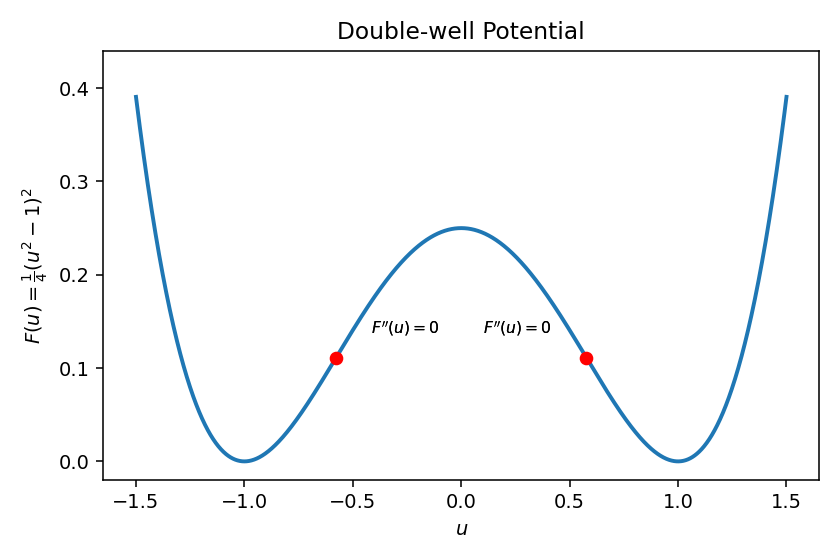}
	\end{subfigure}\hfill
	\begin{subfigure}[t]{0.33\textwidth}
		\centering
		\includegraphics[width=\linewidth]{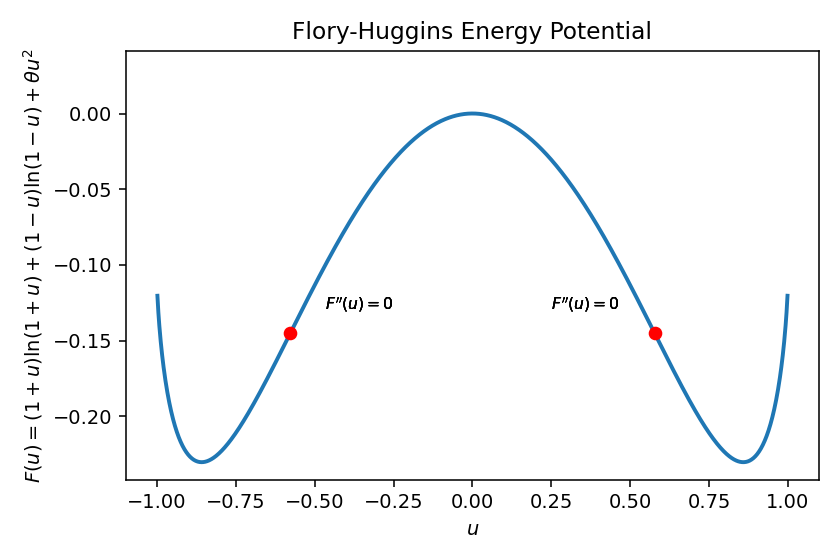}
	\end{subfigure}\hfill
	\begin{subfigure}[t]{0.33\textwidth}
		\centering
		\includegraphics[width=\linewidth]{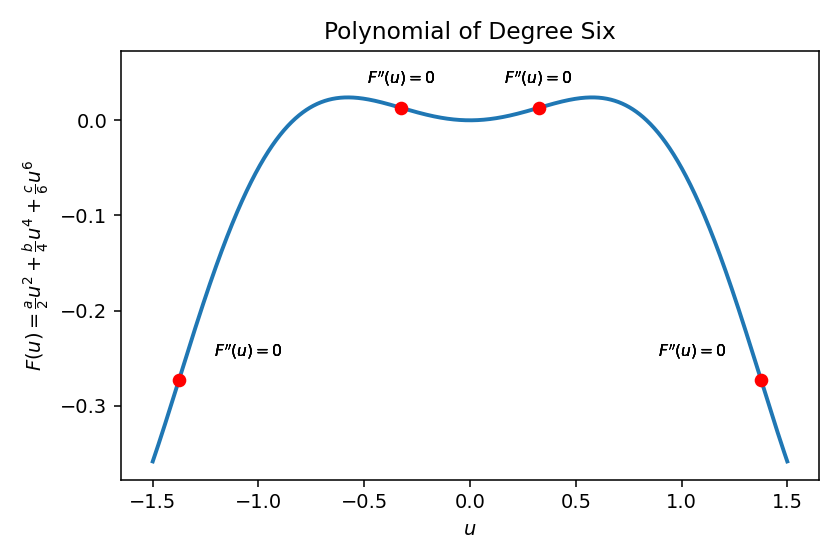}
	\end{subfigure}
	\caption{Three commonly used potential functions, the red $\cdot$ marks the inflection point.}
	\label{fig:potential}
\end{figure}
In what follows, we focus primarily on the double-well potential, while other choices are included to illustrate the broader range of admissible free-energy densities in the CH framework.

\paragraph{Dispersion analysis for the isotropic CH equation}
To derive the dispersion relation, we investigate the evolution of small-amplitude perturbations
around a homogeneous state by setting $u(\mathbf x,t)=u_0+\eta(\mathbf x,t)$ with $|\eta|\ll 1.$
Due to that the mass conservation $u_0=\frac{1}{|\Omega|}\int_\Omega u(\mathbf x,0)\,\mathrm{d}\mathbf x$ is constant. Linearizing CH system~\eqref{eq:CH-iso} around \(u_0\) with constant mobility gives
\begin{equation}\label{eq:CH-linear}
	\eta_t
	=
	M\,\Delta\Bigl(-\Delta\eta+\varepsilon^{-2}F''(u_0)\eta+\beta\Delta^2\eta\Bigr).
\end{equation}

For a Fourier mode $\eta(\mathbf x,t)=\hat\eta_{\mathbf k}(t)e^{i\mathbf k\cdot\mathbf x}$ on a periodic domain, the linearized equation~\eqref{eq:CH-linear} reduces to
\begin{equation}\label{eq:eta}
	\frac{\mathrm d}{\mathrm dt}\hat\eta_{\mathbf k}(t)
	=
	\lambda(\mathbf k;u_0)\,\hat\eta_{\mathbf k}(t),
\end{equation}
where $\lambda(\mathbf k;u_0)=-M\,|\mathbf k|^2\Bigl(|\mathbf k|^2+\varepsilon^{-2}F''(u_0)+\beta|\mathbf k|^4\Bigr)$ denotes the growth rate and  $\mathbf k=(k_x,k_y)$ is the wavevector. If we set $k=|\mathbf k|$ then the dispersion relation can be written as
\[
\lambda(k;u_0)
=
-Mk^2\Bigl(k^2+\varepsilon^{-2}F''(u_0)+\beta k^4\Bigr).
\]

The system~\eqref{eq:eta} is unstable if and only if the growth rate satisfies $\lambda(k;u_0)>0$, i.e., \(k^{2}+\varepsilon^{-2}F''(u_0)+\beta k^{4}<0\). Letting $s=k^{2}\ge0$ implies a  quadratic inequality \(\beta s^{2}+s+\varepsilon^{-2}F''(u_0)<0,\)  which  holds  only when $F''(u_0)<0$.
Moreover, for $\beta>0$ the unstable wavenumbers satisfy
$$
0<s<s_{+},\qquad
s_{+}=\frac{-1+\sqrt{\,1-4\beta\,\varepsilon^{-2}F''(u_0)\,}}{2\beta},
$$
or equivalently,
$$
0<|k|<\sqrt{s_{+}}
=\left(\frac{-1+\sqrt{\,1-4\beta\,\varepsilon^{-2}F''(u_0)\,}}{2\beta}\right)^{1/2}.
$$
In the limit $\beta\to0^{+}$, the above condition reduces to
$0<|k|<\varepsilon^{-1}\sqrt{-F''(u_0)}$.
Therefore, two regimes arise:
\begin{itemize}
	\item[(1)] If $F''(u_0)>0$, the system is stable; all perturbations decay.
	\item[(2)] If $F''(u_0)<0$, the system is spinodally unstable; small perturbations grow exponentially.
\end{itemize}

For the double-well potential, one has \(F''(u_0)=3u_0^2-1\). Hence the spinodal interval is \(|u_0|<1/\sqrt{3}\). In this regime, the linear growth rate \(\lambda=\lambda(\mathbf{k};u_0)\) is positive on a finite band of wave numbers and nonpositive otherwise. For fixed \(u_0\), the dependence of \(\lambda(\mathbf{k};u_0)\) on \(\mathbf{k}\) is illustrated in Fig.~\ref{fig:dispersion-iso}. Here \(u_0=0\) is chosen in the spinodal region, i.e. \(F''(u_0)<0\), so that the growth rate is positive on a finite band of wave numbers.
\begin{figure}[htbp]
	\centering
	\begin{subfigure}[t]{0.5\textwidth}
		\centering
		\includegraphics[width=\linewidth]{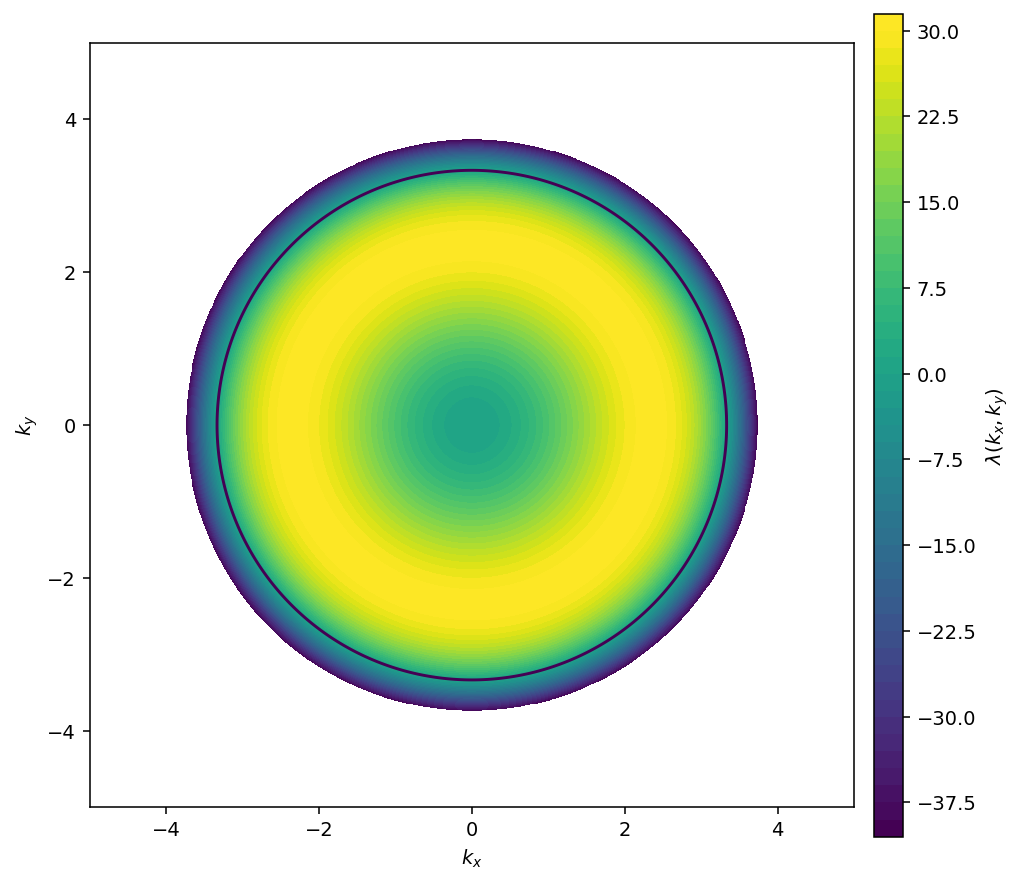}
		\caption{Contour of $\lambda(\mathbf{k})$}
	\end{subfigure}\hfill
	\begin{subfigure}[t]{0.45\textwidth}
		\includegraphics[width=\linewidth]{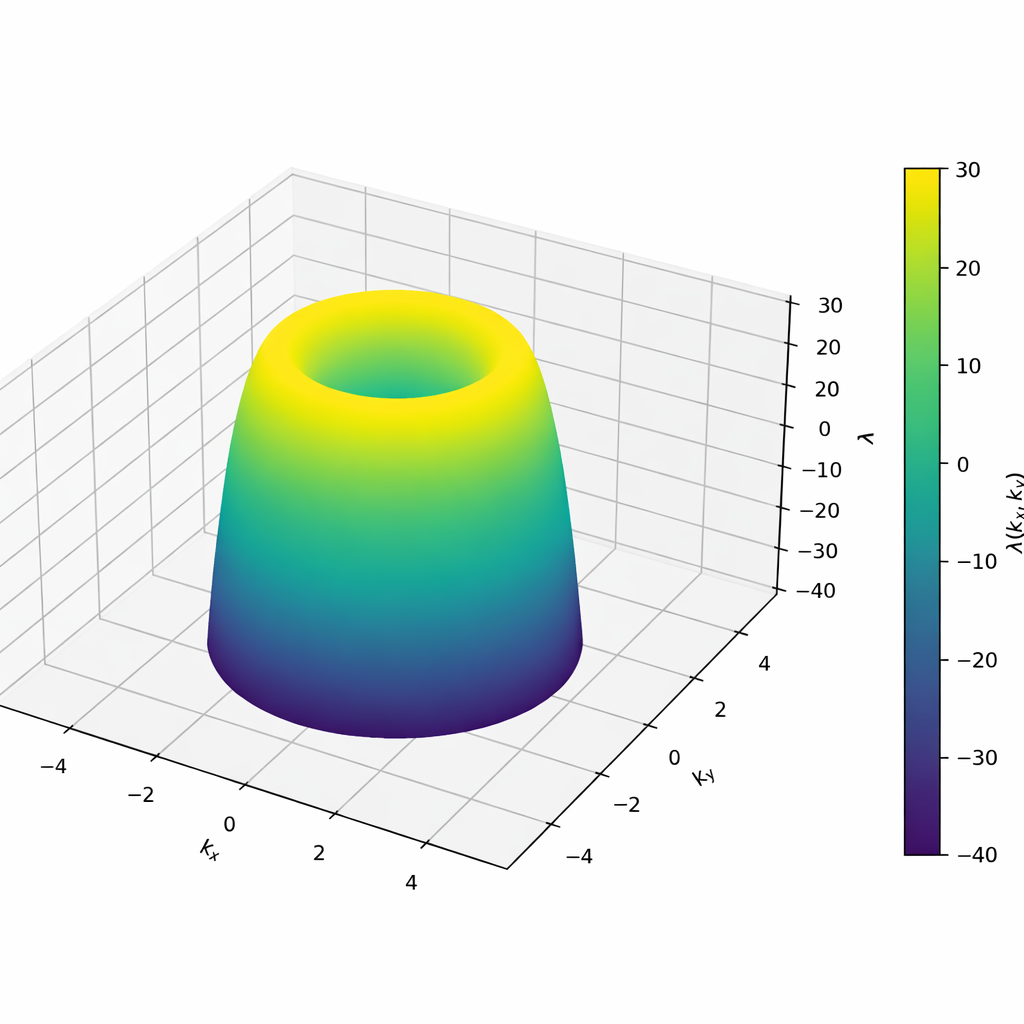}
		\caption{3D surface of $\lambda(\mathbf{k})$}
	\end{subfigure}
	\caption{Dispersion relation of the linearized isotropic Cahn--Hilliard equation for \(u_0=0\).}
	\label{fig:dispersion-iso}
\end{figure}

Consider the following system with periodic boundary conditions
\begin{equation}\label{eq:pde}
	u_t = a\,\Delta u + b\,\Delta^2 u + c\,\Delta^3 u
	\qquad \text{in } \Omega \times (0,\infty),
\end{equation}
where \(a,b,c\in\mathbb{R}\) and \(\Omega=[0,L]^2\). Note that the linearized isotropic CH equation~\eqref{eq:CH-linear} corresponds to \eqref{eq:pde} when
$a=-M\epsilon^{-2}F''(u_0), b=-M, c=M\beta.$ Then we have the following proposition.
\begin{proposition}\label{proposition:disp_cont_disc_refined}
	Consider Eq.~\eqref{eq:pde}. It	admits normal-mode solutions of the form $u(\bm x,t)=U_0\,e^{\lambda t}e^{\mathrm i \mathbf k\cdot \bm x}$, $\mathbf k\in\frac{2\pi}{L}\mathbb Z^2$, \(U_0\in\mathbb C\) where corresponding growth rate \(\lambda\) is given by
	\begin{equation}\label{eq:dis}
		\lambda(\mathbf k)=-a\,k^2+b\,k^4-c\,k^6.
	\end{equation}

\end{proposition}
\begin{proof}
	Substitute $u(\bm x,t)=U_0\,e^{\lambda t}\,e^{\mathrm i \mathbf k\cdot \bm x}$ into  \eqref{eq:pde}, the time derivative yields
	\[
	u_t=\lambda u.
	\] For the spatial operators, the eigenfunction property of the Fourier mode gives $\Delta u = -k^2 u$, and consequently $\Delta^2 u = k^4 u$ and $\Delta^3 u = -k^6 u$. Inserting these expressions into the PDE leads to the characteristic equation
	\[
	\lambda(\mathbf k)=-a\,k^2+b\,k^4-c\,k^6.
	\]
	This completes the proof.
	\end{proof}

In this paper, we discretize the spatial variables using a Fourier spectral method. For system~\eqref{eq:pde}, this yields an ODE for each Fourier mode $\hat{\eta}(\mathbf{k},t)$:
\begin{equation}
	\frac{d}{dt}\hat{\eta}(\mathbf{k},t) = \lambda(\mathbf{k})\hat{\eta}(\mathbf{k},t),
\end{equation}
where the growth rate $\lambda(\mathbf{k})$ is given by~\eqref{eq:dis}.
\begin{proposition}\label{proposition:DispDisc}
	Let $\Delta t$ denote the time step. Applying  Runge--Kutta method to system~\eqref{eq:eta} yields
	\begin{align}\label{eq:DisDis}
		e^{\Delta t \tilde{\lambda}(\mathbf{k})} = R\bigl(\Delta t \lambda(\mathbf{k})\bigr),
	\end{align}
	where $\tilde{\lambda}(\mathbf{k})$ is the discrete growth rate,  $R(z)$ is the stability function.
\end{proposition}
\begin{proof}
	Applying the Runge-Kutta method to system~\eqref{eq:eta} leads to
	\begin{align}\label{eq:DisSemiODE}
		{\hat u}^{n+1}=R(\Delta t\lambda(\mathbf{k})){\hat u}^{n}.
	\end{align}
	Upon introducing the discrete growth rate $\tilde{\lambda}$ via ${\hat u}^{n}=U_0\exp(\Delta t \tilde{\lambda})$, Eq.~\eqref{eq:DisSemiODE} reduces to Eq.~\eqref{eq:DisDis}. This completes the proof of this proposition.
\end{proof}

It follows from Proposition~\ref{proposition:DispDisc}, the discrete amplification factors for explicit Euler ($\rm EE$), implicit Euler ($\rm IE$), and implicit midpoint methods ($\rm IM$) can be calculated by
\begin{equation}\label{eq:g_three_methods}
	\begin{aligned}
		g_{\rm EE}(\mathbf{k})
		&=1-\Delta t\,M|\mathbf{k}|^2\Bigl(A_0+|\mathbf{k}|^2+\beta|\mathbf{k}|^4\Bigr),\\[0.5ex]
		g_{\rm IE}(\mathbf{k})
		&=\frac{1}{1+\Delta t\,M|\mathbf{k}|^2\Bigl(A_0+|\mathbf{k}|^2+\beta|\mathbf{k}|^4\Bigr)},\\[0.5ex]
		g_{\rm IM}(\mathbf{k})
		&=\frac{1-\frac{\Delta t}{2}M|\mathbf{k}|^2\Bigl(A_0+|\mathbf{k}|^2+\beta|\mathbf{k}|^4\Bigr)}
		{1+\frac{\Delta t}{2}M|\mathbf{k}|^2\Bigl(A_0+|\mathbf{k}|^2+\beta|\mathbf{k}|^4\Bigr)}.
	\end{aligned}
\end{equation}
The exact amplification factor is $g_{\rm ex}(\mathbf{k}) = \exp(\Delta t \lambda(\mathbf{k}))$. The discrete amplification maps produced by the three schemes, together with the exact amplification map, are shown below, where the solid curve indicates the contour level $g(\mathbf{k}) = 1$.

\begin{figure}[htbp]
	\centering
	
	\resizebox{0.85\textwidth}{!}{%
		\begin{minipage}{\textwidth}
			\centering
			
			\begin{subfigure}{0.49\linewidth}
				\centering
				\includegraphics[width=\linewidth]{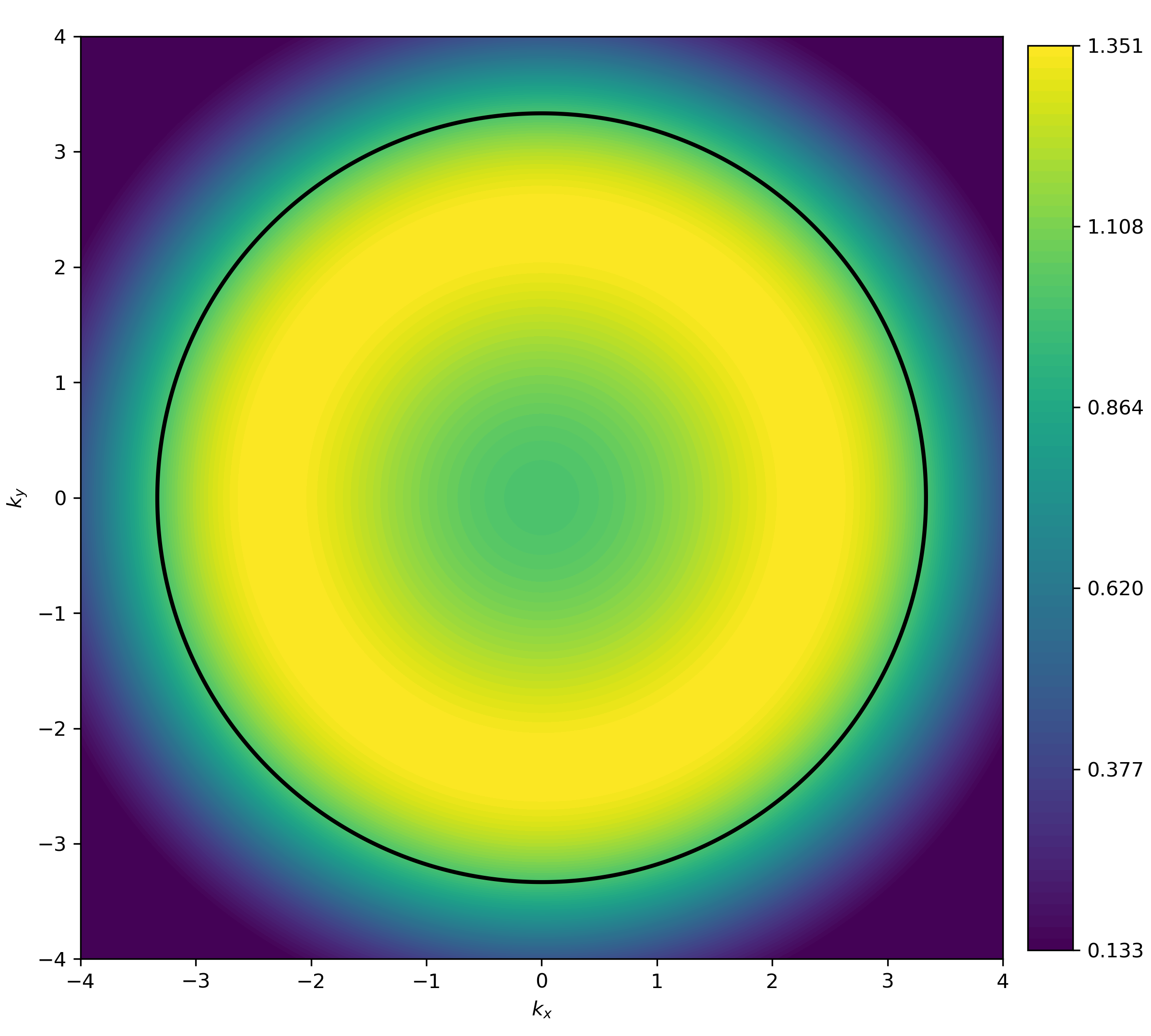}
			\end{subfigure}\hfill
			\begin{subfigure}{0.49\linewidth}
				\centering
				\includegraphics[width=\linewidth]{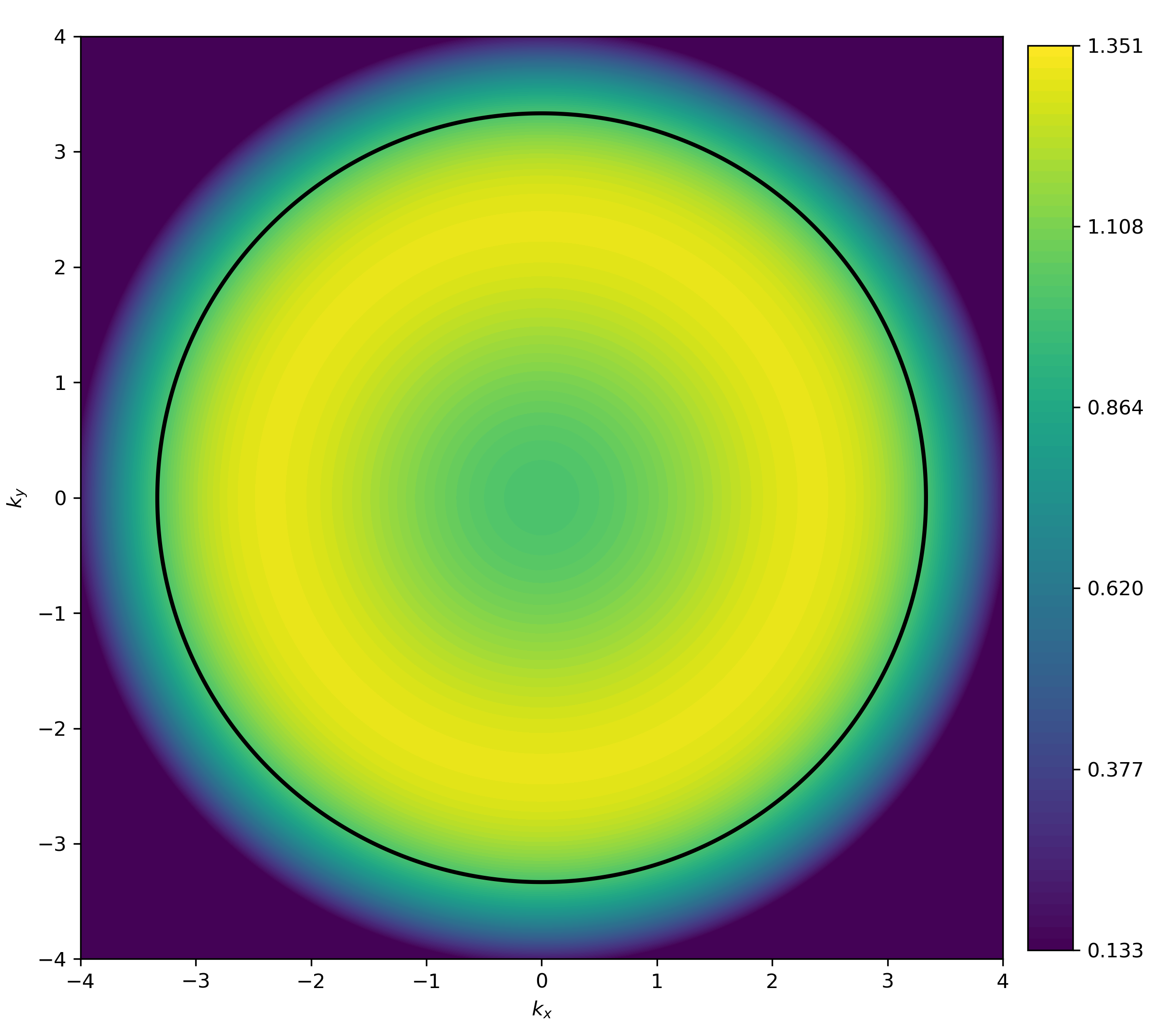}
			\end{subfigure}
			
			\vspace{0.6em}
			
			\begin{subfigure}{0.49\linewidth}
				\centering
				\includegraphics[width=\linewidth]{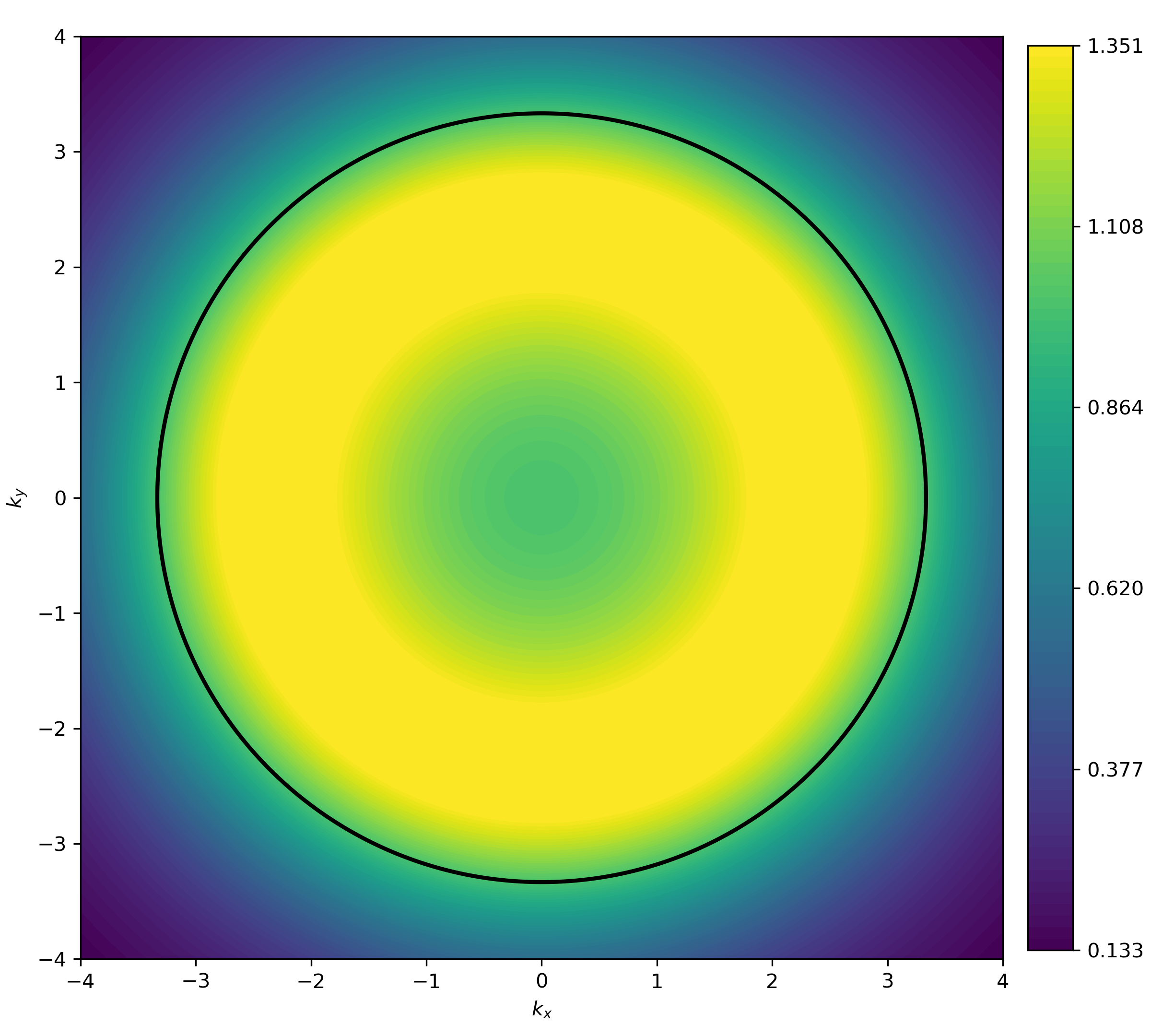}
			\end{subfigure}\hfill
			\begin{subfigure}{0.49\linewidth}
				\centering
				\includegraphics[width=\linewidth]{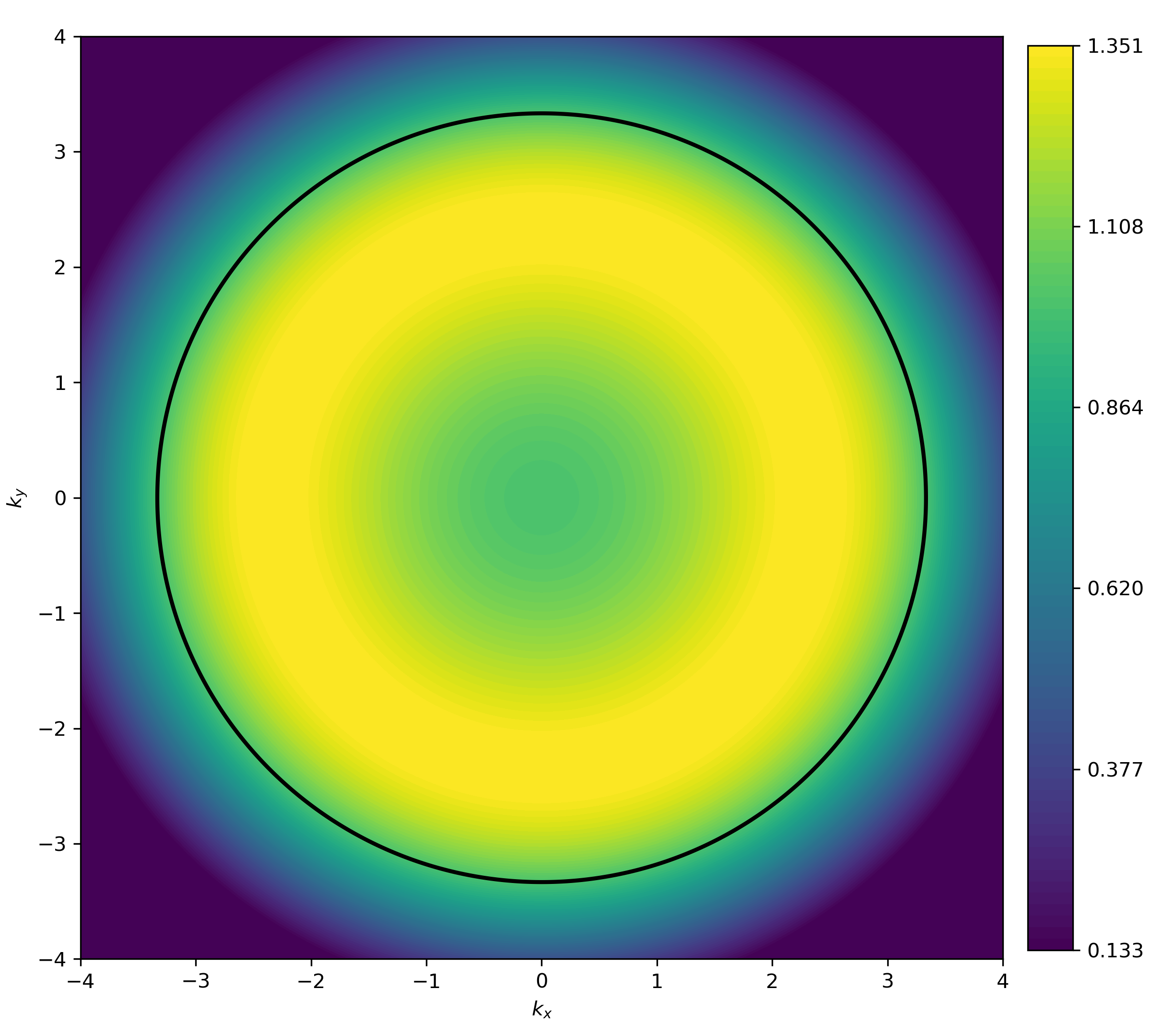}
			\end{subfigure}
			
		\end{minipage}%
	}
	
	\caption{Contours of $g(\mathbf{k})$ in the $(k_x, k_y)$-plane for the isotropic case. Top left: exact; top right: explicit Euler; bottom left: implicit Euler; bottom right: implicit midpoint.}
	\label{fig:comparison-iso}
\end{figure}

As shown in Fig.~\ref{fig:comparison-iso}, the explicit Euler scheme may produce  $|g_{\rm EE}(\mathbf{k})|>1$ for large wavenumbers $|\mathbf{k}|$, causing high-frequency modes to become unstable. In contrast, the implicit Euler scheme is stable for any $\Delta t$, but it excessively dampens the solution. The implicit midpoint scheme is a much closer match to the exact map, providing a second-order accurate and A-stable.  As shown in the plot, it better preserves the shape of the level sets in the $(k_x,k_y)$ plane.

\paragraph{Spinodal instability}
From the dispersion relation derived above, we obtain the growth rate of each Fourier mode. This relation allows us to analyze the dependence of $\lambda(\mathbf{k})$ on the homogeneous background state $u_0$, and hence to distinguish the spinodal regime from the stable regime. 
Spinodal instability occurs when a band of wavenumbers yields $\lambda(\mathbf{k})>0$. This causes the growth of small perturbations.

We consider three homogeneous background states, namely $u_0 = 0,\ 0.70,\ 0.95$, and set the initial data as $\bar u = u_0 + \delta \xi$ with $\xi \sim \mathcal{N}(0,1)$ at each grid point.

Fig.~\ref{fig:spinodal} illustrates the spinodal instability criterion for the isotropic CH model. The homogeneous state is linearly unstable if and only if $F''(u_0) < 0$. In our tests, $u_0 = 0$ yields $F''(u_0) = -1 < 0$; consequently, random perturbations are amplified, and phase separation patterns emerge. In contrast, for $u_0 = 0.70$ and $u_0 = 0.95$, we have $F''(u_0) > 0$, so $\lambda(k) \le 0$ for all $k$. Hence, the perturbations decay and the solution remains nearly homogeneous.

\begin{figure*}
\centering
\includegraphics[width=\linewidth]{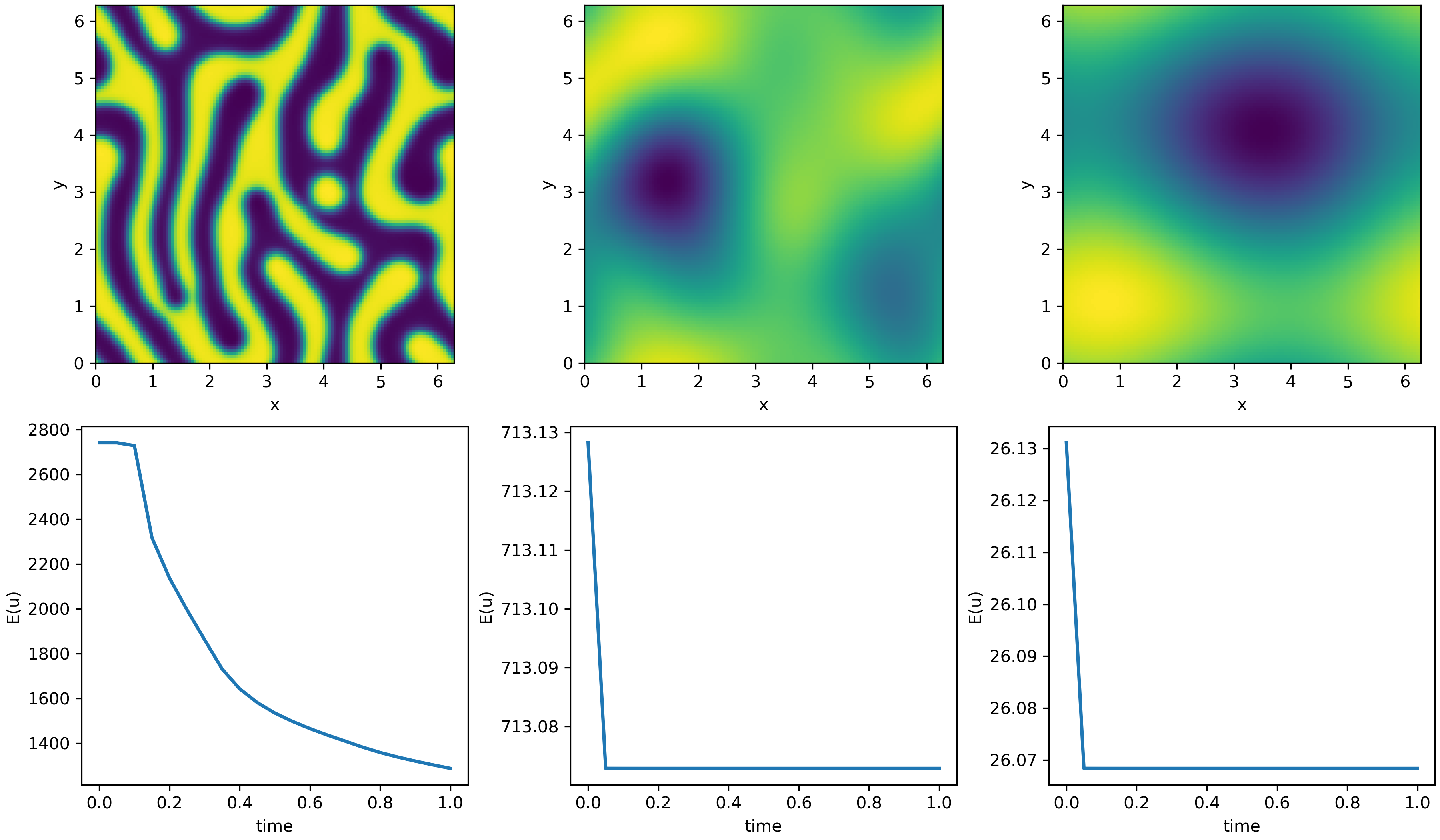}
\caption{Phase separation and energy decay with different initial conditions.}
\label{fig:spinodal}
\end{figure*}

\paragraph{Coarsening law for Isotropic CH model} 
In multi-phase dynamical systems, the coarsening law describes a power-law relationship. Below, we present the coarsening behavior for the isotropic CH model~\eqref{eq:CH-iso}.

\begin{theorem}\label{theorem:coarsening_iso_CH}
	Consider the isotropic CH equation~\eqref{eq:CH-iso} on a periodic domain $\Omega$ with constant mobility $M>0$. Let $\bar E(t)=\frac{E(t)}{|\Omega|}$ denote the energy density.  In the late-stage coarsening regime, the dynamics are governed by a single length scale $R(t)$, and the energy is dominated by interfacial contributions, yielding the scaling law
	\[
	\bar E(t)\sim \frac{\sigma}{R(t)},
	\]
	where \(\sigma>0\) denotes the surface tension. Then
	\[
	R(t)\sim (M\sigma t)^{1/3},
	\qquad
	\bar E(t)\sim M^{-1/3}\sigma^{2/3}t^{-1/3}.
	\]
\end{theorem}
\begin{proof}
	Differentiating the scaling relation $\bar E(t)\sim \sigma/R(t)$ leads to
	\begin{align}\label{eq:1}
		\frac{\mathrm d \bar E}{\mathrm d t}\sim -\frac{\sigma}{R(t)^2}\,\frac{\mathrm d R}{\mathrm d t}.
	\end{align}
	The Gibbs-Thomson curvature relation gives \(\mu\sim \sigma\kappa\), and under the single-length-scale hypothesis \(\kappa\sim 1/R(t)\), hence \(\mu\sim \sigma/R(t)\). Because $\mu$ varies over a length scale of order $R(t)$ in the late-stage regime, it follows that $|\nabla\mu| \sim \frac{\mu}{R(t)} \sim \frac{\sigma}{R(t)^2}$.
	Thus, the energy dissipation per unit volume satisfies
	\begin{align}\label{eq:2}
		-\frac{\mathrm d \bar E}{\mathrm d t}
		=\frac{M}{|\Omega|}\int_\Omega |\nabla\mu|^2\,\mathrm d x
		\sim M\Bigl(\frac{\sigma}{R(t)^2}\Bigr)^2.
	\end{align}
	Combining Eqs.~\eqref{eq:1} and~\eqref{eq:2} yields 
	$$
	\frac{\mathrm d R}{\mathrm d t}\sim \frac{M\sigma}{R(t)^2}.
	$$
	Integrating this scaling ODE gives $R(t)\sim (M\sigma\,t)^{1/3}$, and substituting back into $\bar E(t)\sim \sigma/R(t)$ obtains  $\bar E(t)\sim M^{-1/3}\sigma^{2/3}t^{-1/3}.$ This completes the proof of this theorem.
\end{proof}

We take a random initial condition on \(\Omega=[0,2\pi]^2\), with values uniformly distributed in \([-1,1]\). We simulate system~\eqref{eq:CH-iso} with $\varepsilon=0.3$ and $M=1$ using the new numerical method developed in this paper. Figure~\ref{fig:coarsening-rate} shows that the energy density \(\bar E(t)\) exhibits the predicted \(t^{-1/3}\) decay, and the fitted line has slope \(-1/3\) in the log--log plot.

\begin{figure}[htbp]
\centering
\includegraphics[width=0.7\linewidth]{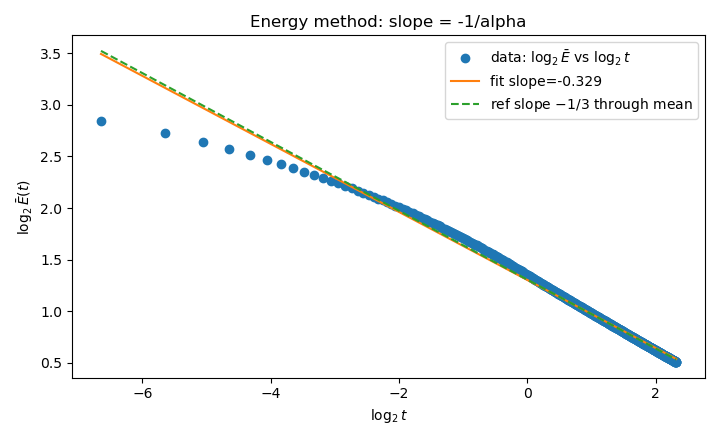}
\caption{Coarsening rate of the isotropic CH equation with $M=1$.}
\label{fig:coarsening-rate}
\end{figure}

\paragraph{Dispersion analysis for the  anisotropic CH equation}
In many crystalline applications, isotropic surface energy proves inadequate for describing faceted morphologies and directionally preferred growth. This motivates the investigation of the anisotropic CH model~\eqref{eq:CH-ani} and its associated dispersion behavior.

Similar to the isotropic case, we linearize the anisotropic model about a homogeneous state. Each Fourier mode $\hat{\eta}(\mathbf{k},t)$ satisfies the scalar ODE~\eqref{eq:eta}, where  
$
\lambda(\mathbf k)
= -M k^2 \Gamma(\theta_{\mathbf{k}})
\bigl(k^2 + \varepsilon^{-2}F''(u_0) + \beta k^4\bigr),
\label{eq:disp-ani}
$ with $
\Gamma(\theta_{\mathbf{k}})
=1+\alpha\,\frac{k_x^4-6k_x^2k_y^2+k_y^4}{(k_x^2+k_y^2)^2},
\mathbf k\neq \mathbf 0.
$

Since the anisotropy factor $\Gamma(\theta_{\mathbf{k}})$ is nonnegative, the zeros of $\lambda(\mathbf{k})$ are determined by $k^2 + \varepsilon^{-2}F''(u_0) + \beta k^4 = 0$. In particular, for $\beta = 0$ equation~\eqref{eq:disp-ani} reduces to
$$
\lambda(\mathbf k)=-k^2\,\Gamma(\theta_{\mathbf k})\bigl(k^2+\varepsilon^{-2}F''(u_0)\bigr),
$$
which gives  the spinodal criterion while  $F''(u_0)<0$.

If $\beta>0$, set $s=k^2\ge 0$, then
$$
\lambda((\mathbf k))=-s\Gamma(\theta_{\mathbf k})\bigl(\varepsilon^{-2}F''(u_0)+s+\beta s^2\bigr).
$$
When $F''(u_0)>0$, one has $\lambda<0$ for all $k$, hence in this case the system is linear stable. If $F''(u_0)<0$, there exists a positive root $s_+>0$ given by
$$
s_{+}=
\frac{-1+\sqrt{1-4\beta\,\varepsilon^{-2}F''(u_0)}}{2\beta}.
$$

Thus, the system is unstable while $0<k<\sqrt{s_+}$, which is the similar to the isotropic case. In Fig.~\ref{fig:dispersion-ani}, we demonstrate the contour plot of $\lambda(\mathbf k)$ for a given orientation (here without loss of generality take $\theta=0$), and $u_0=0$. Here \(u_0\) is chosen in the spinodal region, so that the unstable band of wave numbers is most clearly visible. The orientation is fixed at \(\theta=0\) in order to provide a representative directional slice of the anisotropic dispersion relation.
\begin{figure}[htbp]
\centering
\begin{subfigure}[t]{0.5\textwidth}
	\centering
	\includegraphics[width=\linewidth]{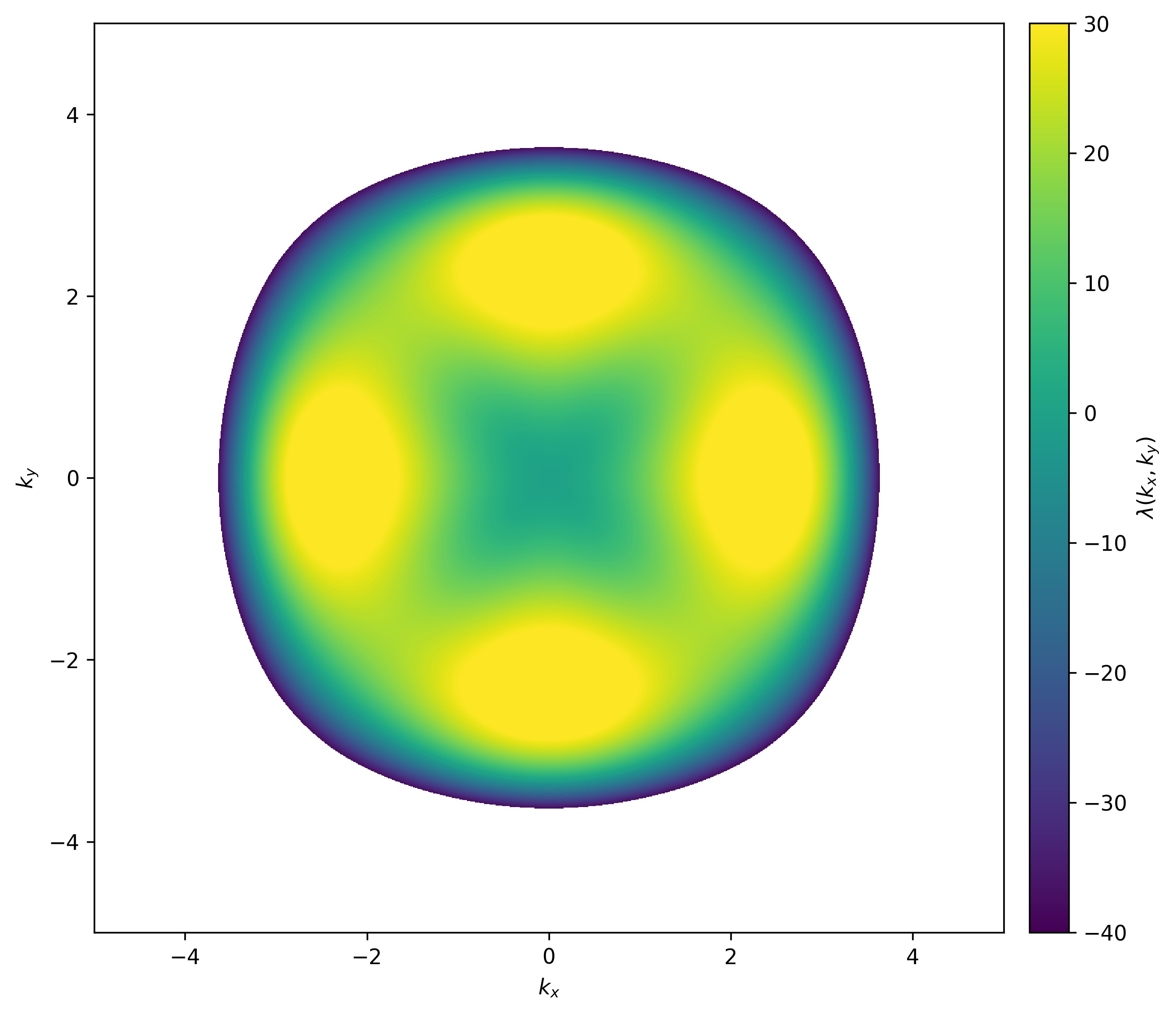}
\end{subfigure}\hfill
\begin{subfigure}[t]{0.45\textwidth}
	\includegraphics[width=\linewidth]{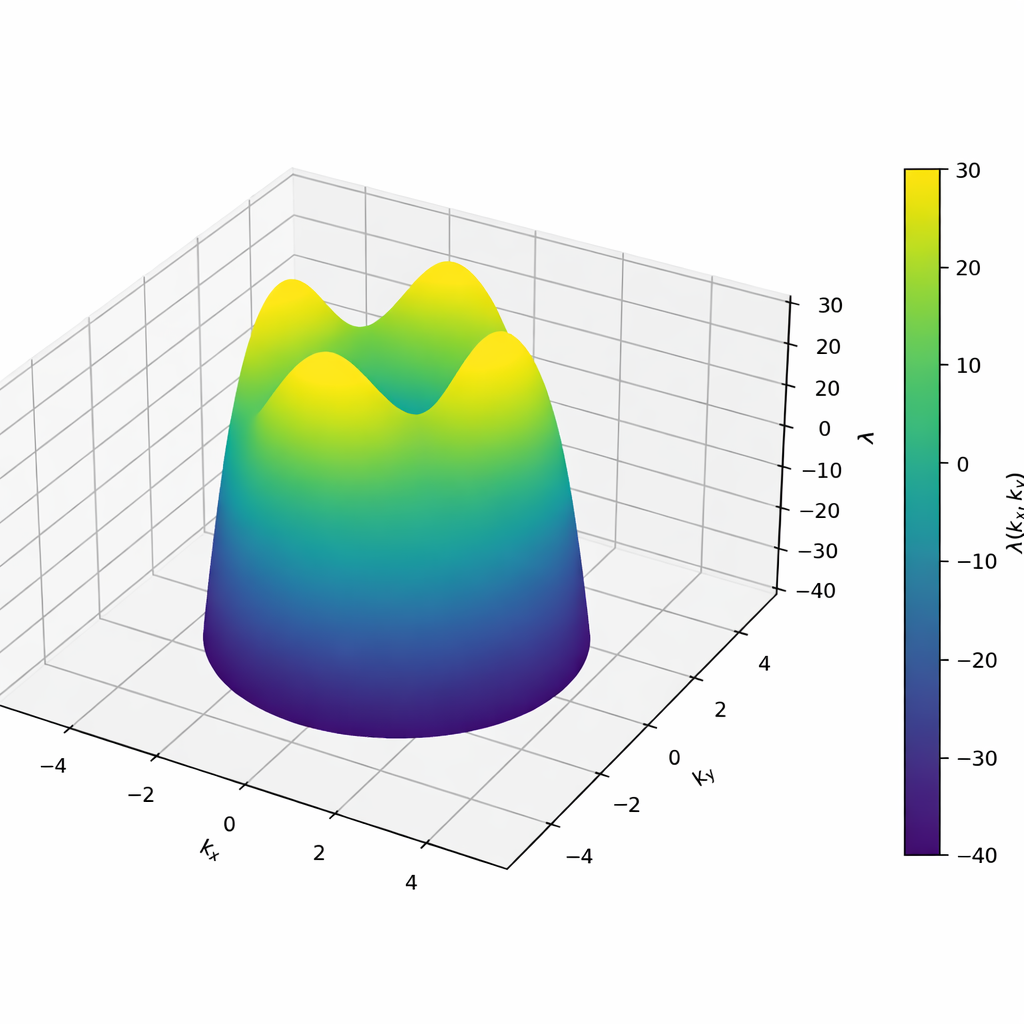}
\end{subfigure}
\caption{Dispersion relation of the linearized anisotropic CH equation with $u_0=0$ and $\theta=0$. Left: Contour plot of $\lambda(\mathbf{k})$;
	Right: 3D surface of $\lambda(\mathbf{k})$.}
\label{fig:dispersion-ani}
\end{figure}

Similar to the isotropic case, the equation~\eqref{eq:eta} yields the exact amplification factor $g_{\rm ex}(\mathbf{k})=\exp\!\bigl(\Delta t\,\lambda(\mathbf{k})\bigr)$.
Applying explicit Euler (EE), implicit Euler (IE), and implicit midpoint (IM) to this mode ODE yields the discrete amplification factors

\begin{equation}\label{eq:g_aniso_three_methods} 
	\begin{aligned} 
		g_{\rm EE}(\mathbf k) &=1-\Delta t\,M k^2 \Gamma(\theta_{\mathbf{k}}) \bigl(k^2+\varepsilon^{-2}F''(u_0)+\beta k^4\bigr),\\[0.5ex] g_{\rm IE}(\mathbf k) &=\frac{1}{1+\Delta t\,M k^2 \Gamma(\theta_{\mathbf{k}}) \bigl(k^2+\varepsilon^{-2}F''(u_0)+\beta k^4\bigr)},\\[0.5ex] g_{\rm IM}(\mathbf k) &=\frac{1-\frac{\Delta t}{2}M k^2 \Gamma(\theta_{\mathbf{k}}) \bigl(k^2+\varepsilon^{-2}F''(u_0)+\beta k^4\bigr)} {1+\frac{\Delta t}{2}M k^2 \Gamma(\theta_{\mathbf{k}}) \bigl(k^2+\varepsilon^{-2}F''(u_0)+\beta k^4\bigr)}.
	\end{aligned} 
\end{equation}
In Fig.~\ref{fig:aniso_amp_2x2}, we compare the exact amplification factor $g_{\rm ex}$ with the numerical factors $g_{\rm EE}$, $g_{\rm IE}$, and $g_{\rm IM}$ on the $(k_x,k_y)$ plane.

\begin{figure}[htbp]
	\centering
	
	\resizebox{0.85\textwidth}{!}{%
		\begin{minipage}{\textwidth}
			\centering
			
			\begin{subfigure}{0.44\linewidth}
				\centering
				\includegraphics[width=\linewidth]{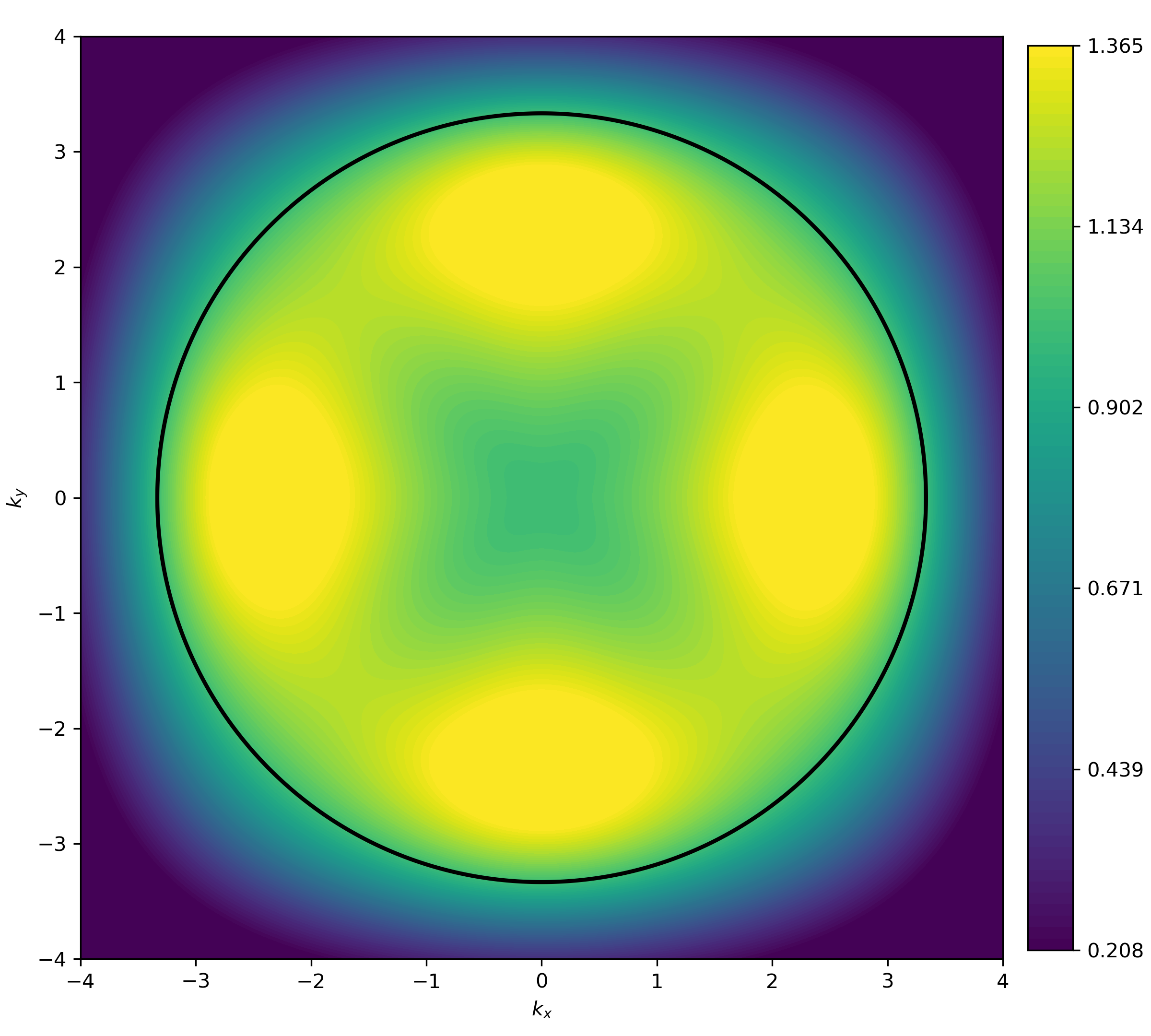}
			\end{subfigure}\hfill
			\begin{subfigure}{0.44\linewidth}
				\centering
				\includegraphics[width=\linewidth]{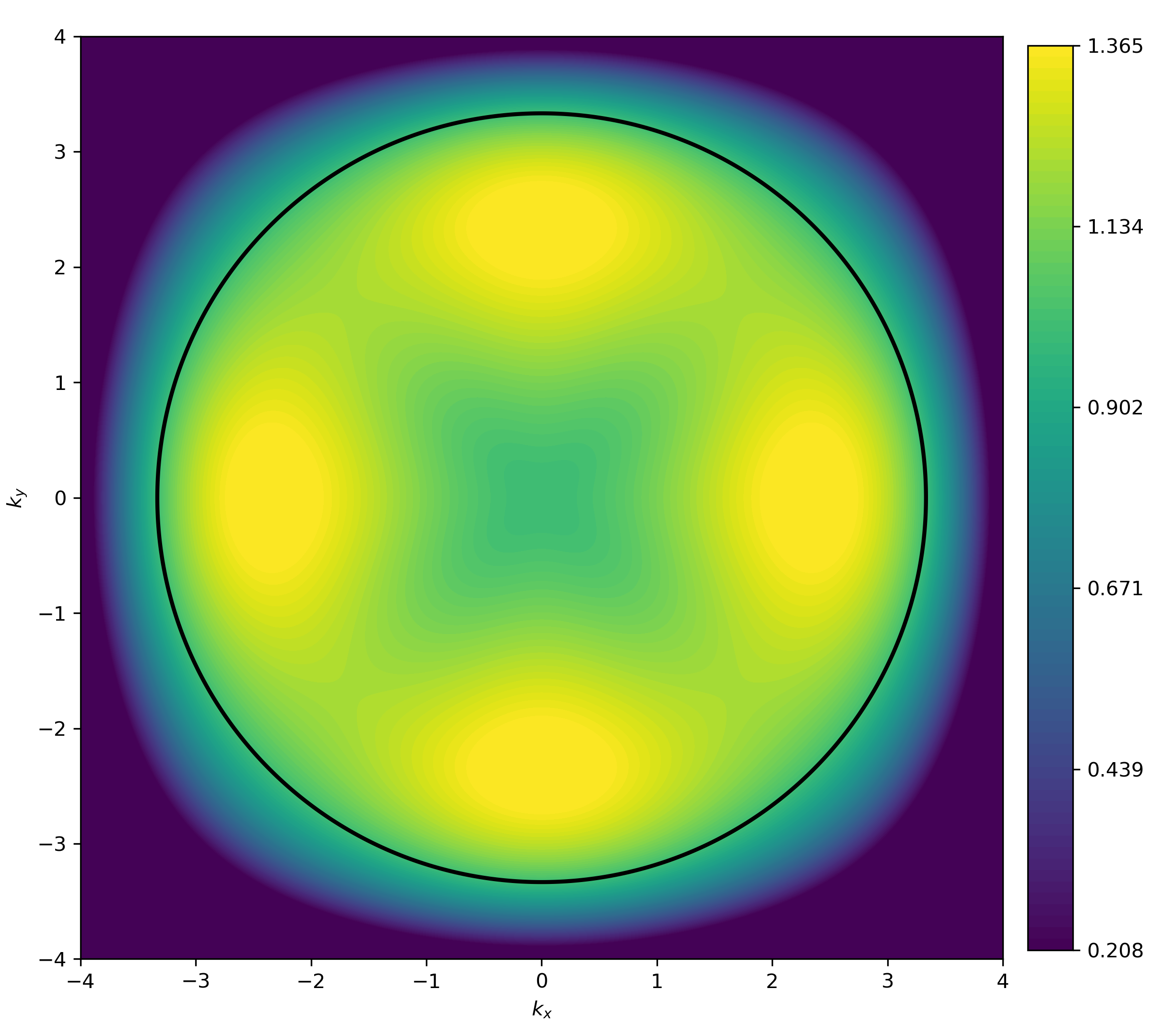}
			\end{subfigure}
			
			\vspace{0.4em}
			
			\begin{subfigure}{0.44\linewidth}
				\centering
				\includegraphics[width=\linewidth]{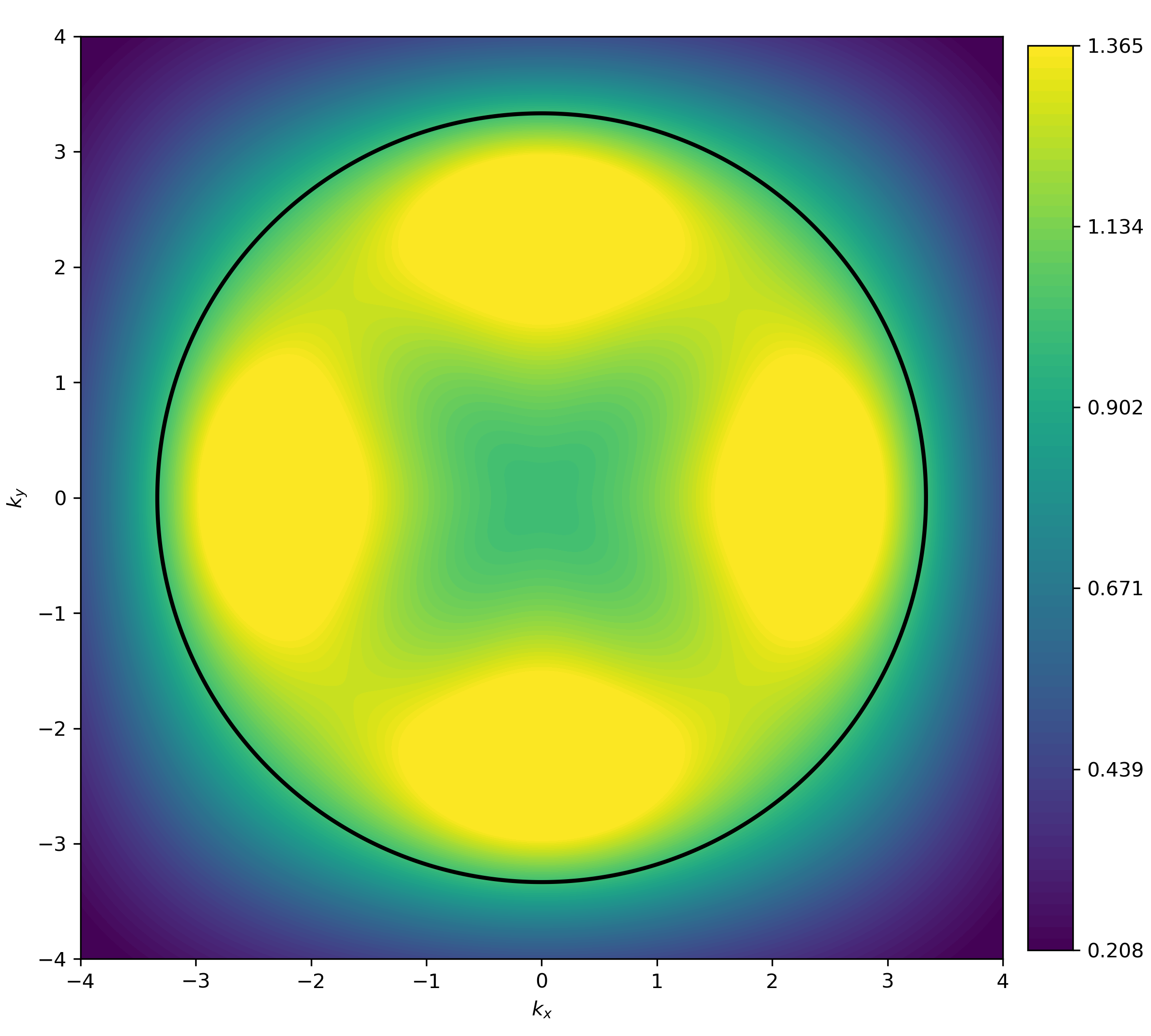}
			\end{subfigure}\hfill
			\begin{subfigure}{0.44\linewidth}
				\centering
				\includegraphics[width=\linewidth]{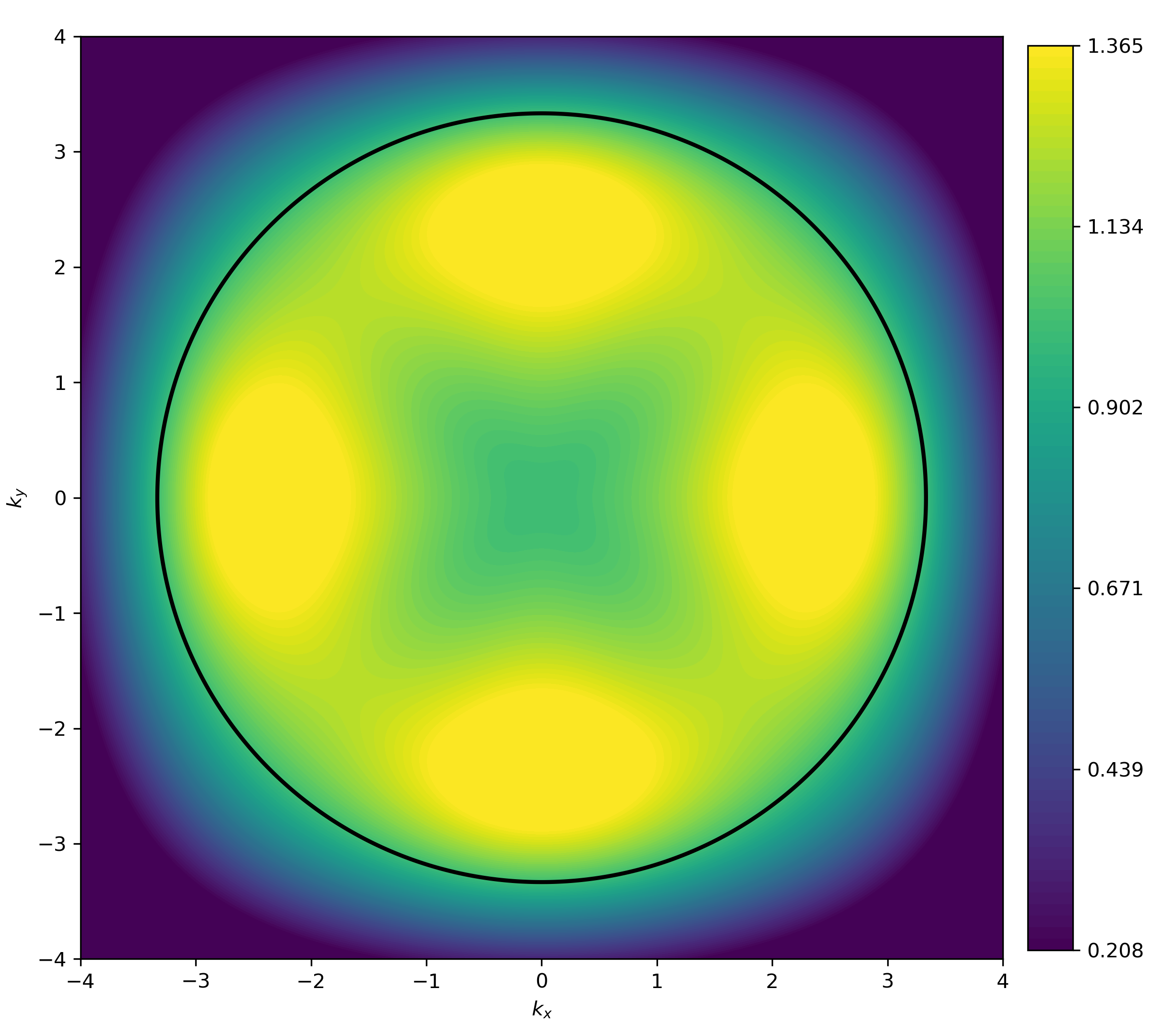}
			\end{subfigure}
			
		\end{minipage}%
	}
	
	\caption{Discrete dispersion relation of the linearized anisotropic CH equation with $u_0=0$ and $\theta=0$. Top left: exact amplification factor; top right: explicit Euler; bottom left: implicit Euler; bottom right: midpoint method.}
	\label{fig:aniso_amp_2x2}
\end{figure}

\paragraph{Equilibrium shapes and missing orientations}
The Gibbs--Thomson relation describes the dependence of the chemical potential on the curvature of an interface. For a 2D interface,  it is described as
\[
\mu=\mu_0+v_m\bigl(\Gamma(\theta)+\Gamma_{\theta\theta}(\theta)\bigr)\,K,
\] where  $\Gamma(\theta)$ is the anisotropic surface energy,  $\mu$ is the interfacial chemical potential.  At equilibrium, $\mu \equiv \mu_e$ is constant, and the equilibrium crystal corresponds to the Wulff shape of $\Gamma(\theta)$. A regular Wulff shape (with no missing orientations) requires the surface stiffness to be nonnegative, which implies
\begin{equation}\label{eq:stiffness}
\Gamma(\theta)+\Gamma_{\theta\theta}(\theta)\ge 0.
\end{equation}
If Eq.~\eqref{eq:stiffness} is violated over an angular interval, the associated high-energy orientations are missing from the equilibrium interface. Geometrically, the naive parametric Wulff curve then contains unstable branches, often called "ears" \cite{CahnHoffman1974,Herring1951,HoffmanCahn1972,Sekerka2005}. The physical Wulff shape is then obtained by convexifying the construction and eliminating these branches.

For the twofold anisotropy, \(\Gamma(\theta;\alpha)=1+\alpha\cos(2\theta)\), so that \(\Gamma+\Gamma_{\theta\theta}=1-3\alpha\cos(2\theta)\), and the corresponding critical value is \(\alpha_c=\frac13\). For the fourfold anisotropy, \(\Gamma(\theta;\alpha)=1+\alpha\cos(4\theta)\), so that \(\Gamma+\Gamma_{\theta\theta}=1-15\alpha\cos(4\theta)\), and the corresponding critical value is \(\alpha_c=\frac{1}{15}\), where \(\alpha_c\) denotes the critical value at which the stiffness~\eqref{eq:stiffness} first vanishes.
We next study how the anisotropy strength $\alpha$ affects the equilibrium morphology and the missing-orientation behavior. 
The results are illustrated in Fig.~\ref{fig:Alpha2-mix}-- \ref{fig:Alpha4-mix}. It is investigated that different values of $\alpha$ give rise to distinct interface shapes and may induce missing-orientation behavior.
\begin{figure}[htbp]
	\centering
	\begin{subfigure}[t]{0.32\textwidth}
		\centering
		\includegraphics[width=\textwidth]{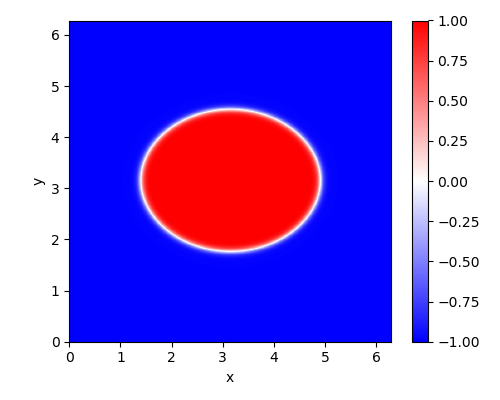}
	\end{subfigure}\hfill
	\begin{subfigure}[t]{0.66\textwidth}
		\centering
		\includegraphics[width=\textwidth]{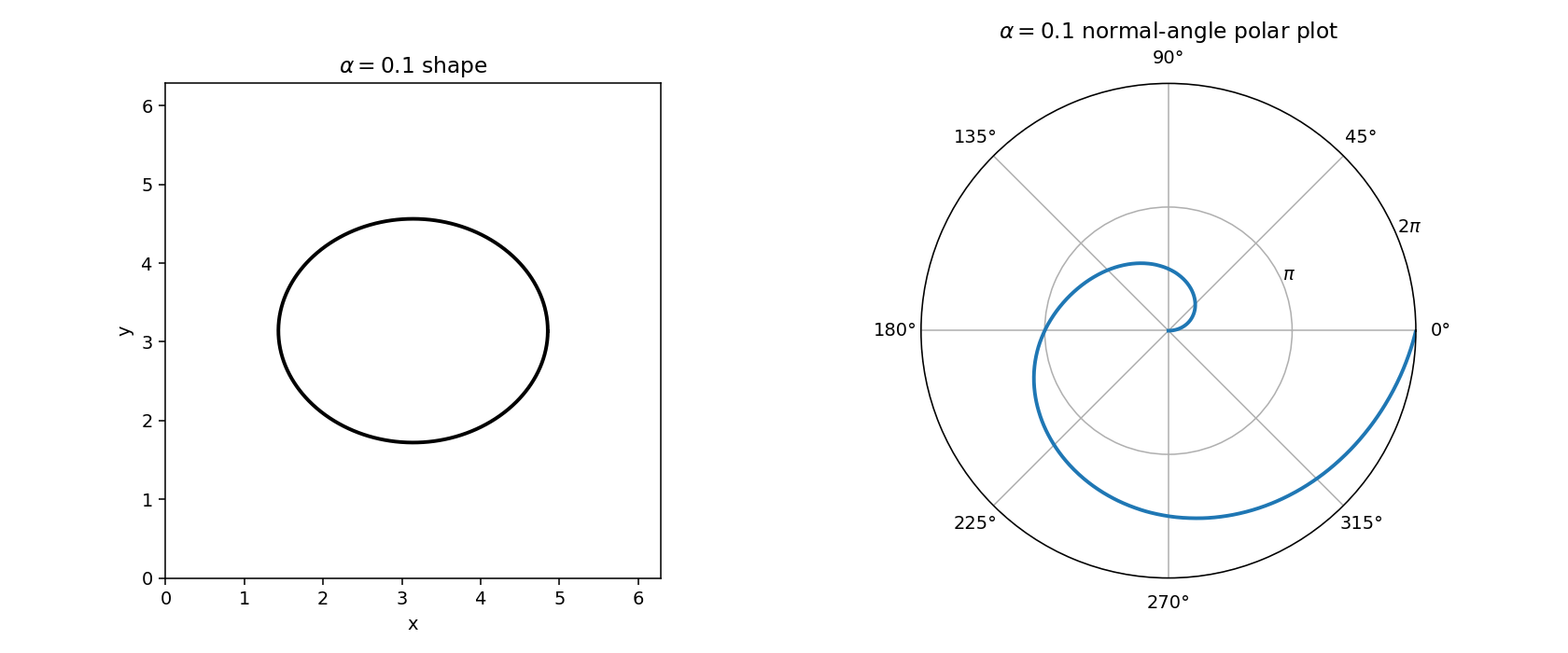}
	\end{subfigure}
	
	\vspace{0.5em}
	
	\begin{subfigure}[t]{0.32\textwidth}
		\centering
		\includegraphics[width=\textwidth]{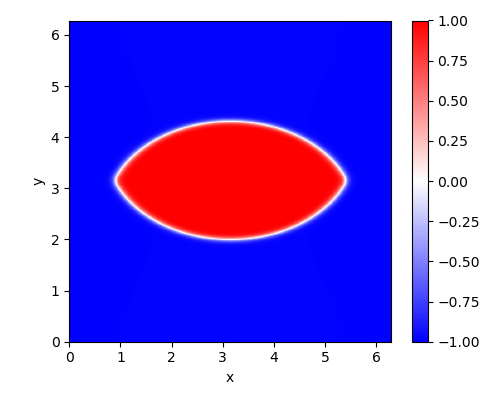}
	\end{subfigure}\hfill
	\begin{subfigure}[t]{0.66\textwidth}
		\centering
		\includegraphics[width=\textwidth]{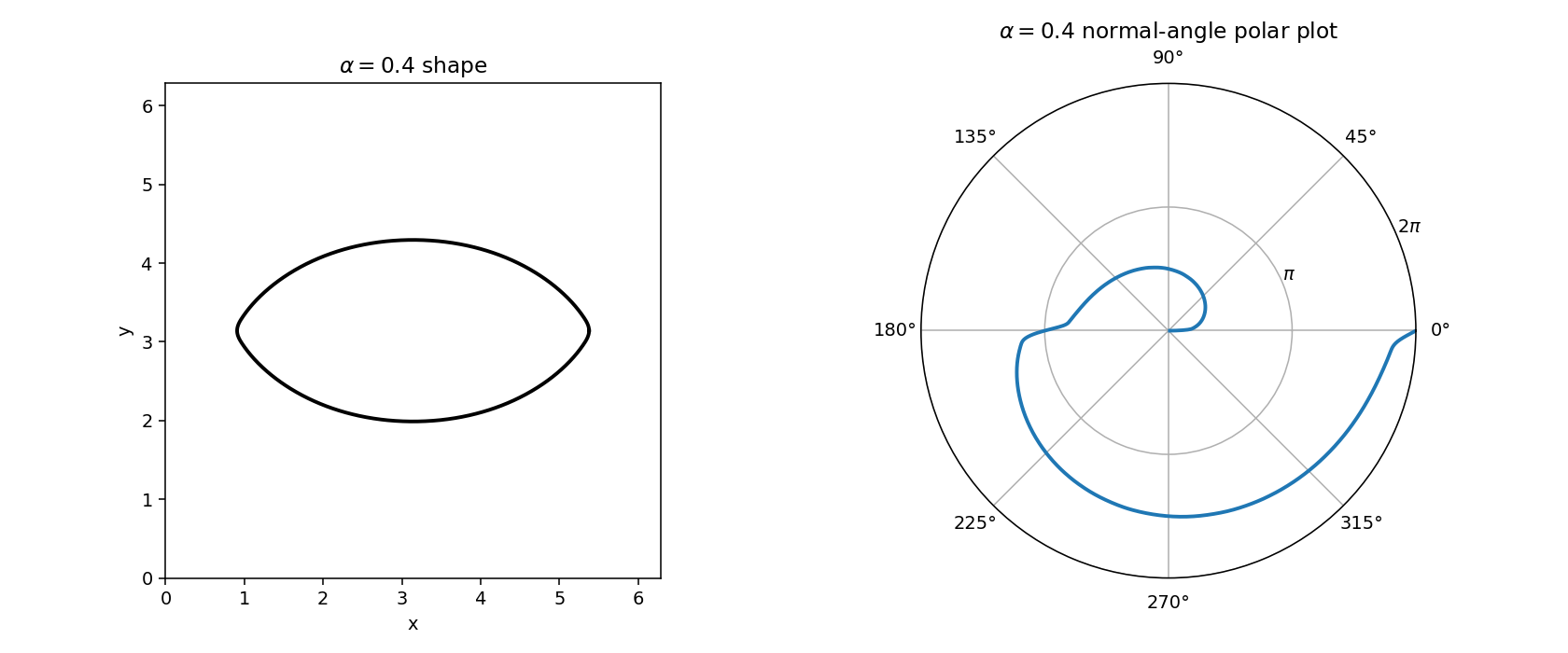}
	\end{subfigure}
	\caption{Morphologies and normal-angle polar plots for the twofold anisotropic CH model.}
	\label{fig:Alpha2-mix}
\end{figure}

\begin{figure}[htbp]
	\centering
	\begin{subfigure}[t]{0.32\textwidth}
		\centering
		\includegraphics[width=\textwidth]{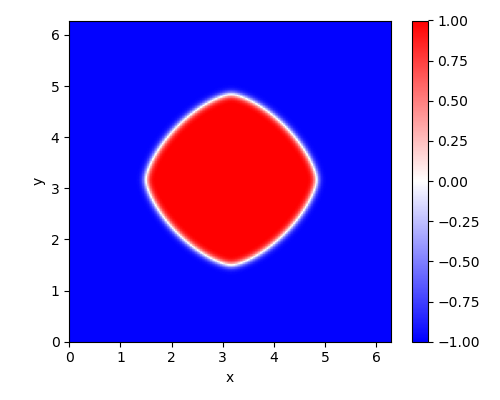}
	\end{subfigure}\hfill
	\begin{subfigure}[t]{0.66\textwidth}
		\centering
		\includegraphics[width=\textwidth]{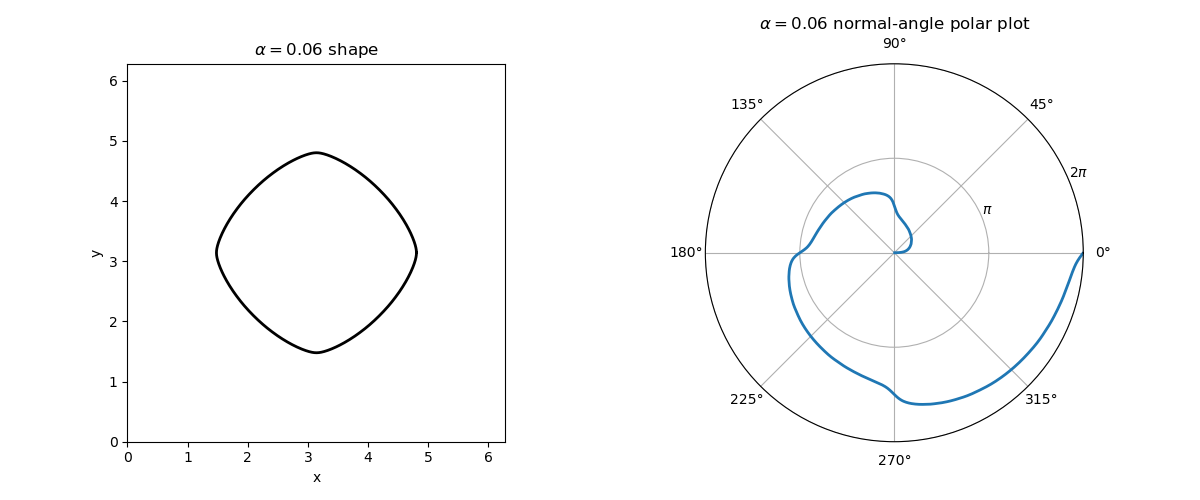}
	\end{subfigure}
	
	\vspace{0.5em}
	
	\begin{subfigure}[t]{0.32\textwidth}
		\centering
		\includegraphics[width=\textwidth]{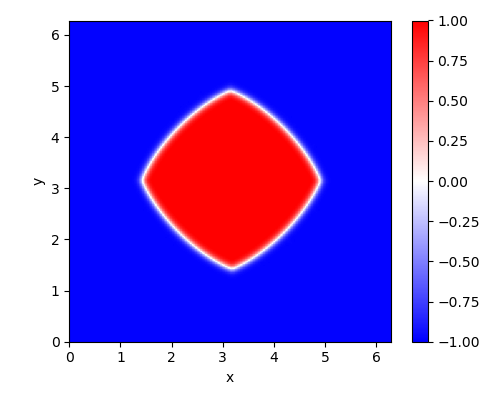}
	\end{subfigure}\hfill
	\begin{subfigure}[t]{0.66\textwidth}
		\centering
		\includegraphics[width=\textwidth]{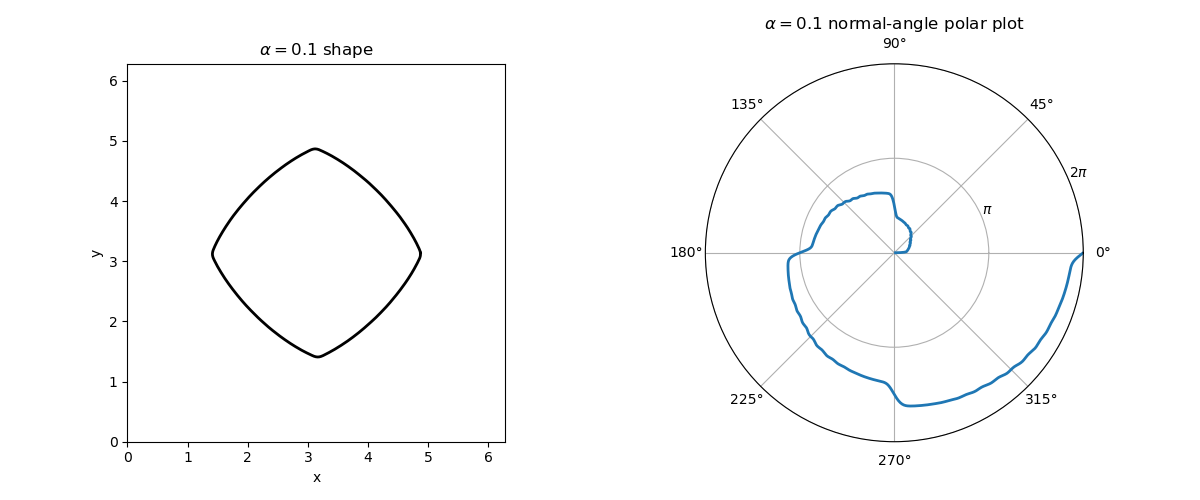}
	\end{subfigure}
	\caption{Morphologies and normal-angle polar plots for the fourfold anisotropic CH model.}
	\label{fig:Alpha4-mix}
\end{figure}
In Fig.~\ref{fig:Alpha2-mix}--\ref{fig:Alpha4-mix}, we demonstrate the interface shapes and their corresponding normal-angle polar plots for different anisotropic factors. It is observed that when Eq.~\eqref{eq:stiffness} is satisfied, the boundary is smooth and the normal-angle polar plot covers the angular range continuously, so that every orientation appears on the interface. In contrast, when Eq.~\eqref{eq:stiffness} is violated, the boundary becomes visibly faceted, and the normal-angle polar plot exhibits a gap-like loss of admissible directions. This indicates that some normal orientations are missing.
\section{Numerical Experiments}
In the previous sections, we derived fully discrete schemes. In this section, we validate the proposed methods through numerical experiments, including temporal convergence tests and simulations of phase evolution of different CH model from different initial conditions.

We first consider the isotropic CH model~\eqref{eq:CH-iso} with double-well potential function and we set the spatial domain to \([0,2\pi]^2\). Take the periodic boundary condition and the two-droplet initial condition:
\begin{equation}\label{eg:two-droplet}
	u_0(x,y)=\max\{u_1(x,y),u_2(x,y)\},
\end{equation}
where $u_i(x,y)=\tanh\!\left(\frac{\sqrt{(x-x_i)^2+(y-y_i)^2}-R_i}{1.2\,\varepsilon}\right)$, with $(x_1,y_1,R_1)=(0.8\pi,\,1.02\pi,\,0.5\pi),$ $(x_2,y_2,R_2)=(1.57\pi,\,0.98\pi,\,0.2\pi).$

The evolution of the phase field and the corresponding discrete energy is presented in Fig.~\ref{fig:iso-double}. As illustrated in these figures, the smaller droplet
gradually dissolves with its mass diffusing toward the larger one, resulting in a single, nearly circular inclusion at later times. The discrete free energy $H(u)$ decays monotonically and levels off, confirming the schemes energy dissipation and convergence toward an equilibrium state.
\begin{figure}[htbp]
	\centering
	\begin{subfigure}[t]{0.55\textwidth}
		\centering
		\includegraphics[width=\textwidth]{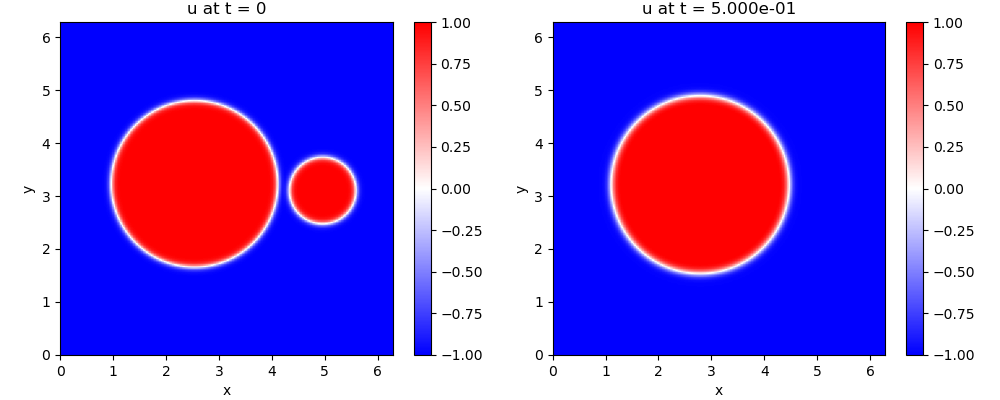}
		\caption{Phase-field evolution for the two-droplet test. Left: initial condition at \(t=0\); right: phase field at \(t=0.5\).}
		\label{fig:iso-double-phase}
	\end{subfigure}
	\hfill
	\begin{subfigure}[t]{0.40\textwidth}
		\centering
		\includegraphics[width=\textwidth]{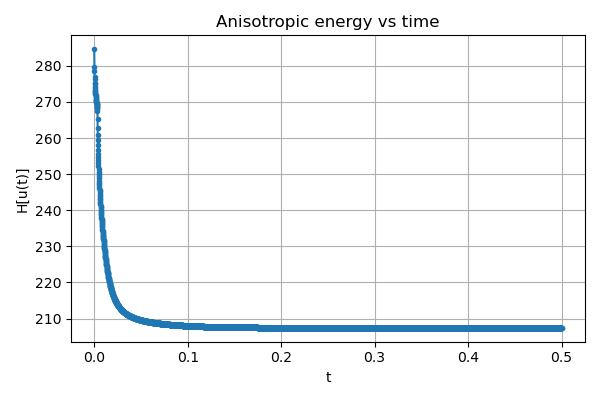}
		\caption{Evolution of discrete energy.}
		\label{fig:iso-double-energy}
	\end{subfigure}
	\caption{Two-droplet test: phase-field evolution and discrete energy decay.}
	\label{fig:iso-double}
\end{figure}

To verify the temporal convergence order, we compute numerical solutions with time step sizes \(4\times10^{-4}\), \(2\times10^{-4}\), and \(1\times10^{-4}\) on a fixed spatial grid, and estimate the order from the error ratios. The results are shown in Fig.~\ref{fig:time-order}.
\begin{figure}[htbp]
\centering
\includegraphics[width=0.65\textwidth]{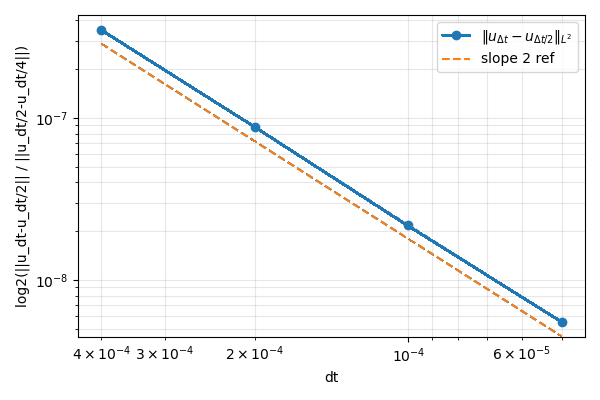}
\caption{Temporal convergence order of the proposed scheme.}
\label{fig:time-order}
\end{figure}

As follows, we also simulate the anisotropic CH model ~\eqref{eq:CH-ani} using the commonly used free-energy functional. Through numerical experiments, we illustrate how anisotropy influences the dynamics. Without loss of generality, in the following examples we set the spatial domain to \([0,2\pi]^2\), the anisotropy strength to $\alpha = 0.1$ and the regularization parameter to $\beta=6\times 10^{-4}$.

\paragraph{Single droplet}
We initialize $u$ by a single droplet centered at $(\pi,\pi)$ which is given by
\[
u_0(x,y)=\tanh\!\Bigl(\frac{\sqrt{(x-\pi)^2+(y-\pi)^2}-0.5\pi}{1.2\varepsilon}\Bigr).
\]
We use a time step size $\Delta t=10^{-4}$. Fig.~\ref{fig:I-single-phase}--\ref{fig:I-single-energy} show the phase-field evolution and the corresponding discrete free energy, respectively.
\begin{figure}[htbp]
	\centering
	\begin{subfigure}[t]{0.55\textwidth}
		\centering
		\includegraphics[width=\textwidth]{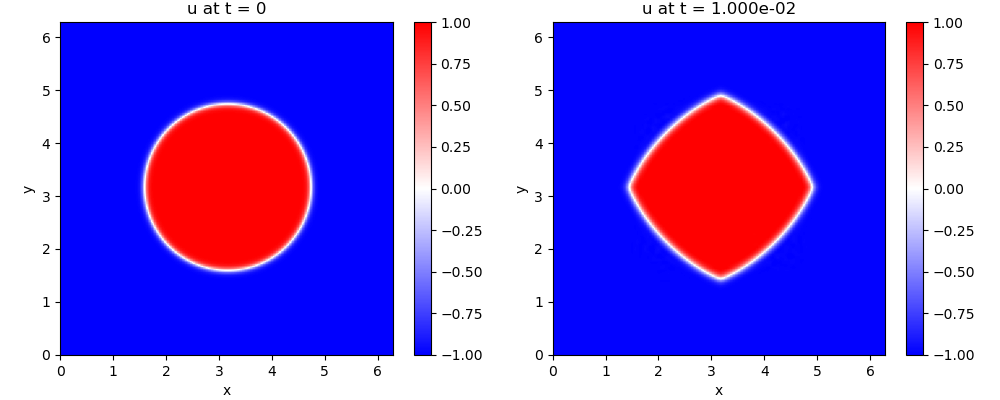}
		\caption{Phase-field evolution for the single-droplet test. Left: initial condition at $t=0$. Right: solution at $t=0.01$.}
		\label{fig:I-single-phase}
	\end{subfigure}
	\hfill
	\begin{subfigure}[t]{0.40\textwidth}
		\centering
		\includegraphics[width=\textwidth]{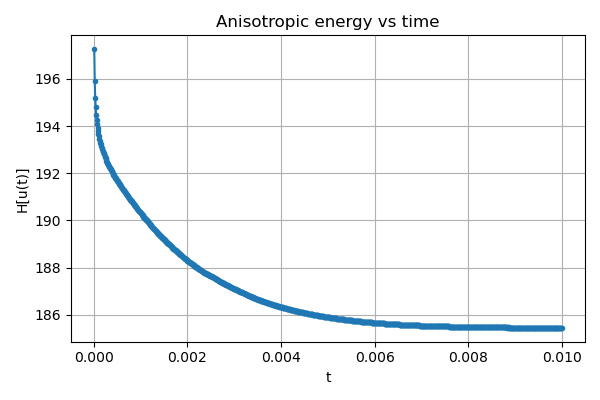}
		\caption{Evolution of discrete energy for the single-droplet test.}
		\label{fig:I-single-energy}
	\end{subfigure}
	\caption{Single-droplet test: phase-field evolution and discrete energy decay.}
	\label{fig:I-single}
\end{figure}

Fig.~\ref{fig:I-single} illustrate the evolution of single droplet. The initially circular interface quickly loses rotational symmetry and develops flat edges and sharp corners. This behavior is expected because the anisotropy function $\Gamma(\theta)$ assigns different interfacial energies to different orientations, so the interface motion favors the directions that minimize the anisotropic surface energy.

Meanwhile, Fig.~\ref{fig:I-single}(b) shows that the discrete free energy decreases monotonically in time. This monotone decay indicates that the proposed scheme preserves the dissipative structure of the CH dynamics.

\paragraph{Double droplets}
We use the two-droplet initial condition in Eq.~\eqref{eg:two-droplet}.
The coalescence dynamics are shown in Figs.~\ref{fig:I-double-phase}--\ref{fig:I-double-energy}.

\begin{figure}[htbp]
\centering
\includegraphics[width=\textwidth]{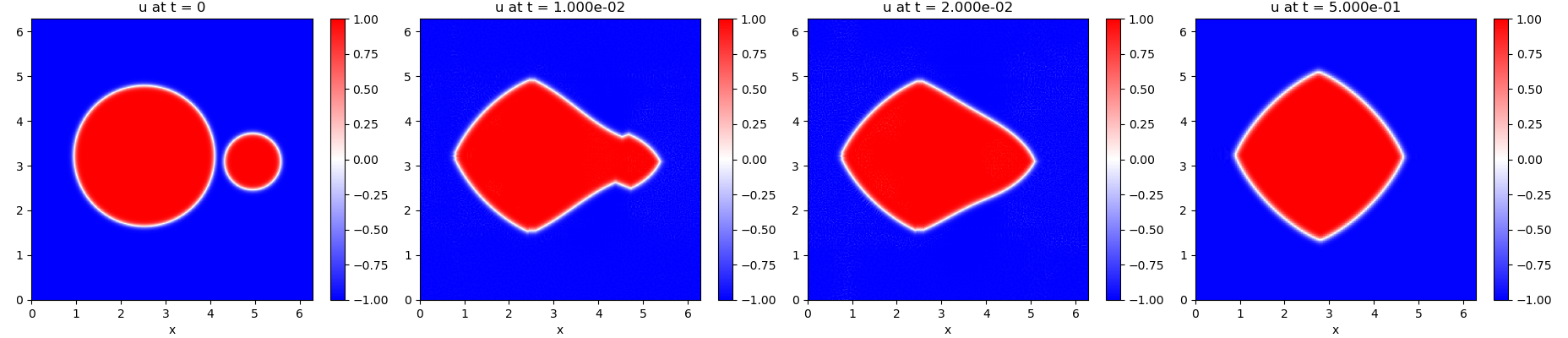}
\caption{Phase-field evolution for the double-droplet test. Left: initial condition at $t=0$. Right: solution at $t=0.5$.}
\label{fig:I-double-phase}
\end{figure}

\begin{figure}[htbp]
\centering
\includegraphics[width=0.6\textwidth]{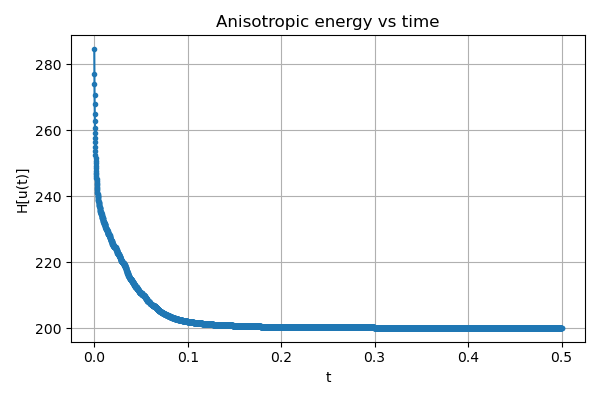}
\caption{Evolution of discrete energy for the double-droplet test.}
\label{fig:I-double-energy}
\end{figure}
Starting from two circular droplets, the solution rapidly develops straight facets and sharp corners aligned with the fourfold preferred directions prescribed by $\Gamma(\theta)$, and the merged droplet evolves toward a diamond-like Wulff shape. Meanwhile, the total mass remains conserved and the discrete free energy decreases monotonically throughout the simulation.

\paragraph{Random initial value}
We next consider a randomly perturbed initial field. We therefore take a randomly perturbed initial field,

\[
u_0(x,y)=\xi(x,y),
\qquad \xi(x,y)\sim \mathcal U(-1,1).
\]
We need to take a smaller time step $\Delta t=10^{-6}$ due to the stronger nonlinearity induced by the orientation-dependent factor, and the results are shown in
Fig.~\ref{fig:I-rand-phase}.

\begin{figure}[htbp]
\centering
\includegraphics[width=0.8\textwidth]{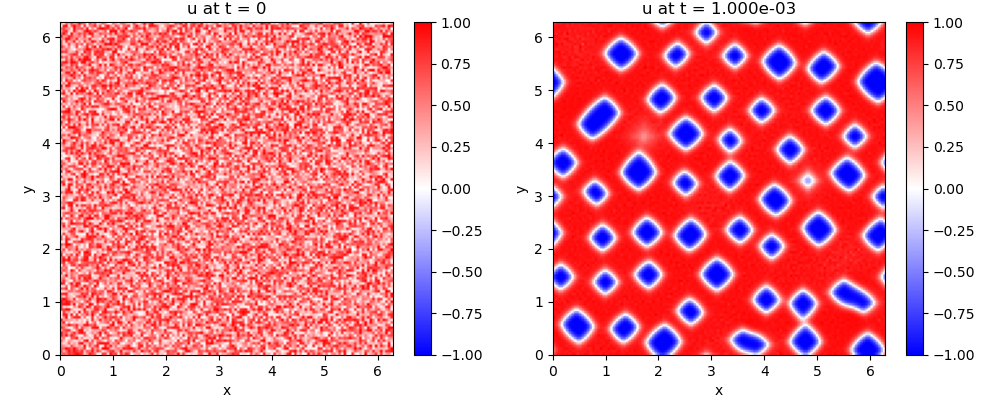}
\caption{Phase-field evolution for the random test. Left: initial condition at $t=0$. Right: solution at $t=0.001$.}
\label{fig:I-rand-phase}
\end{figure}

Fig.~\ref{fig:I-rand-phase} shows that the random perturbations quickly evolve into a fourfold-symmetric pattern, and the interfaces develop diamond-like facets aligned with the preferred orientations prescribed by \(\Gamma(\theta)\).


\section{Concluding remark}
In this paper, we proposed a quadratic reformulation framework for rational-like energy functions. Based on this framework, we developed the Quadratic Conservation Elevation (QCE) method by introducing suitable auxiliary variables and applying the implicit midpoint rule to the corresponding extended system. We have applied this method to Cahn--Hilliard (CH) equations with rational-like free-energy terms and proved that the resulting schemes preserve the original energy dissipation law.

We also analyzed the discrete dispersion relation of the proposed schemes and investigated their consistency with the continuous dynamics. Numerical experiments confirmed the expected spinodal decomposition, coarsening behavior, and second-order temporal accuracy. For anisotropic CH models, the method was further shown to capture missing orientations associated with different anisotropic energy functions. Simulations with various initial conditions illustrated phase separation, long-time coarsening dynamics, and anisotropic evolution.

\appendix
\section{Auxiliary variables for the anisotropic CH model}\label{pf:ch-iso-aux}

The fourfold anisotropy function $\Gamma(\theta)=1+\alpha\cos(4\theta)
=
1-3\alpha+4\alpha\,\frac{u_x^4+u_y^4}{(u_x^2+u_y^2)^2}.$
In order to apply the QCE methodology, we introduce a set of auxiliary variables that quadraticize the nonlinear terms. More precisely, we define $y_1=u_x^2,\,y_2=u_y^2,\,y_3=(y_1+y_2)^2,\,y_4=y_3^{-1},\,y_5=y_1^2+y_2^2,\,y_6=y_4y_5$ together with $z_1=\frac12(u^2-1),\,z_2=z_1^2,\,\phi_1=(\Delta u)^2.$
Here, the variables \(y_1,\dots,y_6\) are introduced to represent the anisotropy coefficient in quadratic form, while \(z_1,z_2\) and \(\phi_1\) correspond to the bulk-potential and higher-order regularization terms. Their discrete counterparts are denoted by \(Y_1,\dots,Y_6\), \(Z_1\), \(Z_2\), and \(\Phi_1\), respectively.

\section{Proof of Theorem~\ref{thm:ani-diss-law}}\label{pf:ani-diss-law}

\begin{proof}
	Let $\Gamma^n=\Gamma(\Theta(U^n))$, calculating the discrete time derivative of energy gives
	\begin{equation}\label{eq:E-diff-mid}
		\small
		\begin{aligned}
			&\frac{E_{\mathrm{ani},h}(U^{n+1})-E_{\mathrm{ani},h}(U^n)}{\Delta t}\\
			=&
			\bigl(\Gamma^{n+\frac12}\odot U_x^{n+\frac12},\delta_t U_x^{n+1}\bigr)_h 
			+\bigl(\Gamma^{n+\frac12}\odot U_y^{n+\frac12},\delta_t U_y^{n+1}\bigr)_h 
			+\frac{1}{\varepsilon^2}\bigl(\Gamma^{n+\frac12}\odot f^{n+\frac12},\delta_t U^{n+1}\bigr)_h\\
			+&\beta\bigl(\Gamma^{n+\frac12}\odot \Delta_hU^{n+\frac12},\,\delta_t(\Delta_hU)^{n+1}\bigr)_h 
			+\frac12\bigl(\delta_t\Gamma^{n+1},\, (U_x^{n+\frac12})^{\odot2}
			+(U_y^{n+\frac12})^{\odot2}\bigr)_h \\
			+&\frac{1}{\varepsilon^2}\bigl(\delta_t\Gamma^{n+1},\,F(U^{n+\frac12})\bigr)_h 
			+\frac{\beta}{2}\bigl(\delta_t\Gamma^{n+1},\,(\Delta_hU^{n+\frac12})^{\odot2}\bigr)_h,
		\end{aligned}
	\end{equation}
	where $f^{n+\frac12}$ satisfies $F(U^{n+1})-F(U^n)=f^{n+\frac12}\odot (U^{n+1}-U^n).$	
	The summation of last three terms can be simplified to $4\alpha\,\bigl(\delta_t Y_6^{n+1},\,\Psi^{n+\frac12}\bigr)_h$,
	where $\Psi^{n+\frac12}	=	\frac12\Bigl((U_x^{n+\frac12})^{\odot2}	+(U_y^{n+\frac12})^{\odot2}\Bigr)
	+\varepsilon^{-2}F\!(U^{n+\frac12})
	+\frac{\beta}{2}\Bigl(\Delta_hU^{n+\frac12}\Bigr)^{\odot2}$.
	
	Set $\mathcal{K}^{n+\frac12}
	=
	Y_5^{n+\frac12}\odot Y_4^{n+1}\odot Y_4^{n}\odot
	(Y_1^{n+\frac12}+Y_2^{n+\frac12})$,
	we can simplify $\delta_t Y_6$ to 
	$\delta_t Y_6^{n+1}=4\,U_x^{n+\frac12}\odot\Bigl(Y_4^{n+\frac12}\odot Y_1^{n+\frac12}-\mathcal{K}^{n+\frac12}
	\Bigr)\odot \delta_t U_x^{n+1} + 4\,U_y^{n+\frac12}\odot\Bigl(Y_4^{n+\frac12}\odot Y_2^{n+\frac12}-\mathcal{K}^{n+\frac12}
	\Bigr)\odot \delta_t U_y^{n+1}$. From Eq.~\eqref{eq:Gamma-aux-h} we have
	\begin{equation}\label{eq:Gamma-aux-hh}
		\small
		\begin{split}
			4\alpha\,\bigl(\delta_t Y_6^{n+1},\,\Psi^{n+\frac12}\bigr)_h	
			=&
			\Bigl(	\underbrace{16\alpha\,\Psi^{n+\frac12}\odot U_x^{n+\frac12}\odot
				\bigl(
				Y_4^{n+\frac12}\odot Y_1^{n+\frac12}-\mathcal{K}^{n+\frac12}
				\bigr)}_{\widetilde B_x^{n+\frac12}},\, \delta_t U_x^{n+1}\Bigr)_h\\
			+&
			\Bigl(	\underbrace{16\alpha\,\Psi^{n+\frac12}\odot U_y^{n+\frac12}\odot
				\bigl(
				Y_4^{n+\frac12}\odot Y_2^{n+\frac12}-\mathcal{K}^{n+\frac12}
				\bigr)}_{\widetilde B_y^{n+\frac12}},\, \delta_tU_y^{n+1}\Bigr)_h
		\end{split}
	\end{equation}
	
	Substituting \eqref{eq:Gamma-aux-hh} into \eqref{eq:E-diff-mid} gives
	\begin{equation}\label{eq:collect_general_h2}
		\begin{aligned}
			\frac{E_{\mathrm{ani},h}(U^{n+1})-E_{\mathrm{ani},h}(U^n)}{\Delta t}
			=&
			(A^{n+\frac12},\,\delta_t U^{n+1}\bigr)_h
			+\bigl(B_x^{n+\frac12},\,\delta_t U_x^{n+1}\bigr)_h\\
			&+\bigl(B_y^{n+\frac12},\,\delta_t U_y^{n+1}\bigr)_h
			+\bigl(C^{n+\frac12},\,\delta_t(\Delta_hU)^{n+1}\bigr)_h,
		\end{aligned}
	\end{equation}
	where $A^{n+\frac12}=\varepsilon^{-2}\Gamma^{n+\frac12}\odot f^{n+\frac12}$, 
	$B_x^{n+\frac12}=\Gamma^{n+\frac12}\odot U_x^{n+\frac12}+\widetilde B_x^{n+\frac12}$,  $B_y^{n+\frac12}=\Gamma^{n+\frac12}\odot U_y^{n+\frac12}+\widetilde B_y^{n+\frac12}$,  
	$C^{n+\frac12}=\beta\,\Gamma^{n+\frac12}\odot \Delta_hU^{n+\frac12}.$
	
	Let $\mu^{n+\frac12}=	A^{n+\frac12}-D_xB_x^{n+\frac12}-B_y^{n+\frac12}D_y^\top+\Delta_hC^{n+\frac12}$, employing integration by part to the last three terms of \eqref{eq:collect_general_h2} gives
	\[
	\frac{E_{\mathrm{ani},h}(U^{n+1})-E_{\mathrm{ani},h}(U^n)}{\Delta t}
	=
	\bigl(\mu^{n+\frac12},\Delta_h\mu^{n+\frac12}\bigr)_h
	=
	-\|D_x\mu^{n+\frac12}\|_h^2-\|\mu^{n+\frac12}D_y^\top\|_h^2
	\le 0.
	\]
	This completes the proof.
\end{proof}
\section{Proof of Theorem~\ref{thm:mann_midpoint}}\label{pf:mann_midpoint}
	
\begin{proof}
We first show that the map \(T\) is a contraction on \(\mathcal C\) under the condition of the theorem. For any \(X,Y\in \mathcal C\),
\[
	T(X)-T(Y)
	=
	\mathcal F^{-1}\!\Bigl[
	\bigl(I-\tfrac12\Delta t\,\widehat{\mathcal L}\bigr)^{-1} 
	\Delta t\Bigl(
	\widehat{\mathcal N\!\left(\frac{X+X^n}{2}\right)}
	-
	\widehat{\mathcal N\!\left(\frac{Y+X^n}{2}\right)}
	\Bigr)
	\Bigr].
\]
Since \(\widehat{\mathcal L}(\bm k)\le 0\), we have $1-\tfrac12\Delta t\,\widehat{\mathcal L}(\bm k)\ge 1$ for all $\bm k.$
Therefore, 
\[
\left\|
\mathcal F^{-1}\!\Bigl[
\bigl(I-\tfrac12\Delta t\,\widehat{\mathcal L}\bigr)^{-1}\widehat f
\Bigr]
\right\|
\le
\left\|\mathcal F^{-1}(\widehat f)\right\|,
\]
by Parseval's identity.	Applying this estimate with $f=\mathcal N\!\left(\frac{X+X^n}{2}\right)-\mathcal N\!\left(\frac{Y+X^n}{2}\right)$, we obtain
	\[
	\begin{aligned}
		\|T(X)-T(Y)\|
		&\le
		\Delta t\,
		\left\|\mathcal N\!\left(\frac{X+X^n}{2}\right)-\mathcal N\!\left(\frac{Y+X^n}{2}\right)\right\|\le
		\frac{\Delta t\,L_{\mathcal N}}{2}\|X-Y\|.
	\end{aligned}
	\]
	Hence \(T\) is a contraction whenever \(\Delta t\,L_{\mathcal N}<2\). The Banach fixed-point theorem then implies that \(T\) has a unique fixed point \(X^\ast\in \mathcal C\).
	
	Moreover, for the Mann iteration,
	\[
	\begin{aligned}
		\|X^{(m+1)}-X^\ast\|
		&\le
		(1-\omega)\|X^{(m)}-X^\ast\|
		+\omega\|T(X^{(m)})-T(X^\ast)\|\\
		&\le
		\bigl(1-\omega+\omega q\bigr)\|X^{(m)}-X^\ast\|,
	\end{aligned}
	\]
	where $	q=\frac{\Delta t\,L_{\mathcal N}}{2}<1.$
	Since \(0<\omega\le 1\), one has \(1-\omega+\omega q<1\), and therefore \(X^{(m)}\to X^\ast\) linearly.
	
	Finally, by construction of the map \(T\), the fixed-point equation \(X^\ast=T(X^\ast)\) is equivalent to the nonlinear system arising from the midpoint discretization. This completes the proof.

\end{proof}

\bibliographystyle{siamplain}
\bibliography{references}

@article{Kobayashi1993,
	author  = {Kobayashi, Ryo},
	title   = {Modeling and numerical simulations of dendritic crystal growth},
	journal = {Physica D: Nonlinear Phenomena},
	year    = {1993},
	volume  = {63},
	number  = {3-4},
	pages   = {410--423},
	doi     = {10.1016/0167-2789(93)90120-P}
}

@article{Torabi2009,
	author  = {Torabi, Solmaz and Lowengrub, John and Voigt, Axel and Wise, Steven},
	title   = {A new phase-field model for strongly anisotropic systems},
	journal = {Proceedings of the Royal Society A: Mathematical, Physical and Engineering Sciences},
	year    = {2009},
	volume  = {465},
	number  = {2105},
	pages   = {1337--1359},
	doi     = {10.1098/rspa.2008.0385}
}

@article{TaylorCahn1998,
	author  = {Taylor, Jean E. and Cahn, John W.},
	title   = {Diffuse interfaces with sharp corners and facets: Phase field models with strongly anisotropic surfaces},
	journal = {Physica D: Nonlinear Phenomena},
	year    = {1998},
	volume  = {112},
	number  = {3-4},
	pages   = {381--411},
	doi     = {10.1016/S0167-2789(97)00177-2}
}

@article{Eggleston2001,
	author  = {Eggleston, John J. and McFadden, Grant B. and Voorhees, Peter W.},
	title   = {A phase-field model for highly anisotropic interfacial energy},
	journal = {Physica D: Nonlinear Phenomena},
	year    = {2001},
	volume  = {150},
	number  = {1-2},
	pages   = {91--103},
	doi     = {10.1016/S0167-2789(00)00222-0}
}

@article{Wulff1901,
	author  = {Wulff, G.},
	title   = {Zur Frage der Geschwindigkeit des Wachsthums und der Aufl{\"o}sung der Krystallfl{\"a}chen},
	journal = {Zeitschrift f{\"u}r Kristallographie},
	year    = {1901},
	volume  = {34},
	number  = {1-6},
	pages   = {449--530},
	doi     = {10.1524/zkri.1901.34.1.449}
}

@article{Boettinger2002,
	author  = {Boettinger, William J. and Warren, James A. and Beckermann, Christoph and Karma, Alain},
	title   = {Phase-Field Simulation of Solidification},
	journal = {Annual Review of Materials Research},
	year    = {2002},
	volume  = {32},
	number  = {1},
	pages   = {163--194},
	doi     = {10.1146/annurev.matsci.32.101901.155803}
}

@incollection{Herring1951,
	author    = {Herring, Conyers},
	title     = {Surface Tension as a Motivation for Sintering},
	booktitle = {Fundamental contributions to the continuum theory of evolving
	phase interfaces in solids},
	publisher = {Springer, Berlin},
	address   = {New York},
	year      = {1999},
	pages     = {33--69}
}

@article{cahn1958,
	author  = {Cahn, John W. and Hilliard, John E.},
	title   = {Free Energy of a Nonuniform System. I. Interfacial Free Energy},
	journal = {The Journal of Chemical Physics},
	year    = {1958},
	volume  = {28},
	number  = {2},
	pages   = {258--267},
	doi     = {10.1063/1.1744102}
}

@article{Yang2017IEQ,
	author  = {Yang, Xiaofeng and Zhao, Jia and Wang, Qiang and Shen, Jie},
	title   = {Numerical approximations for a three-component Cahn--Hilliard phase-field model based on the invariant energy quadratization method},
	journal = {Mathematical Models and Methods in Applied Sciences},
	year    = {2017},
	volume  = {27},
	number  = {11},
	pages   = {1993--2030},
	doi     = {10.1142/S0218202517500373}
}

@article{Shen2018SAV,
	author  = {Shen, Jie and Xu, Jie and Yang, Jiang},
	title   = {The scalar auxiliary variable (SAV) approach for gradient flows},
	journal = {Journal of Computational Physics},
	volume = {353},
	pages = {407-416},
	year = {2018},
	issn = {0021-9991},
	doi = {10.1016/j.jcp.2017.10.021},
}

@book{Furihata2010DVD,
	author    = {Furihata, Daisuke and Matsuo, Takayasu},
	title     = {Discrete Variational Derivative Method: A Structure-Preserving Numerical Method for Partial Differential Equations},
	publisher = {Chapman and Hall/CRC},
	year      = {2010},
	isbn      = {9781420094466},
	doi       = {10.1201/b10387}
}

@article{Lu2025,
	author = {Lu, Nan and Wang, Yushun and Sun, Yajuan},
	title = {Geometric Integration for the Linear-Gradient System},
	journal = {SIAM Journal on Scientific Computing},
	volume = {47},
	number = {1},
	pages = {A46--A71},
	year = {2025},
	doi = {10.1137/23M1617618}
}

@article{Celledoni2012,
	author  = {Celledoni, Elena and Grimm, Volker and McLachlan, Robert I. and McLaren, David I. and O'Neale, Dion and Owren, Brynjulf and Quispel, G. R. W.},
	title   = {Preserving energy resp. dissipation in numerical PDEs using the ``Average Vector Field'' method},
	journal = {Journal of Computational Physics},
	year    = {2012},
	volume  = {231},
	pages   = {6770--6789},
	doi     = {10.1016/j.jcp.2012.06.022}
}

@article{Sekerka2005,
	author  = {Sekerka, Robert F.},
	title   = {Analytical criteria for missing orientations on three-dimensional equilibrium shapes},
	journal = {Journal of Crystal Growth},
	year    = {2005},
	volume  = {275},
	number  = {1},
	pages   = {77--82},
	doi     = {10.1016/j.jcrysgro.2004.10.069},
}

@article{karma1996,
	author  = {Karma, Alain and Rappel, W.-J.},
	title   = {Phase-field method for computationally efficient modeling of solidification with arbitrary interface kinetics},
	journal = {Physical Review E},
	year    = {1996},
	volume  = {53},
	number  = {4},
	pages   = {R3017--R3020},
	doi     = {10.1103/PhysRevE.53.R3017}
}

@article{ma2006implementation,
	author  = {Ma, Ning and Chen, Qing and Wang, Yunzhi},
	title   = {Implementation of high interfacial energy anisotropy in phase field simulations},
	journal = {Scripta Materialia},
	year    = {2006},
	volume  = {54},
	number  = {11},
	pages   = {1919--1924},
	doi     = {10.1016/j.scriptamat.2006.02.005}
}

@article{Mann1953,
	author  = {Mann, W. R.},
	title   = {Mean Value Methods in Iteration},
	journal = {Proceedings of the American Mathematical Society},
	volume  = {4},
	year    = {1953},
	pages   = {506--510},
	doi     = {10.2307/2032162}
}

@article{Lu2024,
	author = {Lu, Nan and Chenxi, Wang and Zhang, Zhen},
	title = {Decoupled and energy stable schemes for phase-field surfactant model based on mobility operator splitting technique},
	year = {2025},
	volume = {459},
	page = {Paper No. 116365, 19},
	journal = {Journal of Computational and Applied Mathematics},
	doi = {10.1016/j.cam.2024.116365}
}

@article{McLachlanQuispelRobidoux1999,
	author  = {McLachlan, Robert I. and Quispel, G. R. W. and Robidoux, Nicolas},
	title   = {Geometric integration using discrete gradients},
	journal = {Philosophical Transactions of the Royal Society of London. Series A: Mathematical, Physical and Engineering Sciences},
	volume  = {357},
	number  = {1754},
	pages   = {1021--1045},
	year    = {1999},
	doi     = {10.1098/rsta.1999.0363}
}

@book{flory1953principles,
	title={Principles of Polymer Chemistry},
	author={Flory, Paul J.},
	year={1953},
	publisher={Cornell University Press}
}

@article{huggins1941solutions,
	title={Solutions of long chain compounds},
	author={Huggins, Maurice L.},
	journal={The Journal of Chemical Physics},
	volume={9},
	number={5},
	pages={440--440},
	year={1941}
}

@article{wise2009energy,
	title={An energy stable and convergent finite-difference scheme for
	the modified phase field crystal equation},
	author={Wang, C. and Wise, S. M.},
	journal={SIAM Journal on Numerical Analysis},
	volume={49},
	number={3},
	pages={945--969},
	year={2011},
	doi= {10.1137/090752675}
}

@article{Wu2022Review,
	author  = {Wu, Hao},
	title   = {A review on the Cahn--Hilliard equation: classical results and recent advances in dynamic boundary conditions},
	journal = {Electronic Research Archive},
	year    = {2022},
	volume  = {30},
	number  = {7},
	pages   = {2788--2832}
}

@article{CahnHoffman1974,
	author  = {Cahn, J. W. and Hoffman, D. W.},
	title   = {A Vector Thermodynamics for Anisotropic Surfaces II. Curved and Faceted Surfaces},
	journal = {Acta Metallurgica},
	volume = {22},
	number = {10},
	pages = {1205-1214},
	year = {1974},
	issn = {0001-6160},
	doi = {10.1016/0001-6160(74)90134-5}
}

@article{HoffmanCahn1972,
	author  = {Hoffman, D. W. and Cahn, J. W.},
	title   = {A Vector Thermodynamics for Anisotropic Surfaces I. Fundamentals and Application to Plane Surface Junctions},
	journal = {Surface Science},
	volume = {31},
	pages = {368-388},
	year = {1972},
	issn = {0039-6028},
	doi = {10.1016/0039-6028(72)90268-3}
}

@article{Tapley2022,
	author  = {Tapley, Benjamin K.},
	title   = {Geometric Integration of ODEs Using Multiple Quadratic Auxiliary Variables},
	journal = {SIAM Journal on Scientific Computing},
	year    = {2022},
	volume  = {44},
	number  = {4},
	pages   = {A2651--A2668},
	doi     = {10.1137/21M1442644}
}

@article{DahlbyOwren2011,
	author  = {Dahlby, Morten and Owren, Brynjulf},
	title   = {A General Framework for Deriving Integral Preserving Numerical Methods for PDEs},
	journal = {SIAM Journal on Scientific Computing},
	year    = {2011},
	volume  = {33},
	number  = {5},
	pages   = {2318--2340},
	doi     = {10.1137/100810174}
}
\end{document}